\documentclass[10pt]{article}
\usepackage{amsmath,amsfonts,amssymb,amsthm,mathrsfs}
\usepackage[a4paper,vmargin={3.5cm,3.5cm},hmargin={2.5cm,2.5cm}]{geometry}
\usepackage[font=sf, labelfont={sf,bf}, margin=1cm]{caption}
\usepackage{graphicx,graphics}
\usepackage{epsfig}
\usepackage{latexsym}
\usepackage[applemac]{inputenc}
\usepackage{ae,aecompl}
\usepackage[english]{babel}
 \usepackage[colorlinks=true]{hyperref}
\usepackage{pstricks}
\usepackage{enumerate}
\usepackage{cleveref}
  \crefname{theorem}{Theorem}{Theorems}
  \crefname{lemma}{Lemma}{Lemmas}
  \crefname{remark}{Remark}{Remarks}
  \crefname{proposition}{Proposition}{Propositions}
  \crefname{definition}{Definition}{Definitions}
  \crefname{corollary}{Corollary}{Corollaries}
  \crefname{section}{Section}{Sections}
  \crefname{figure}{Figure}{Figures}

\newcommand{\one}{{\triangle^{1}}}
\newcommand{\two}{{\triangle^{2}}}
\newcommand{\Z}{\mathbb{Z}}
\renewcommand{\P}{\mathbb{P}}
\newcommand{\E}{\mathbb{E}}

\newtheorem{theorem}{Theorem}[]

\newtheorem{proposition}[theorem]{Proposition}
\newtheorem{lemma}[theorem]{Lemma}
\newtheorem{corollary}[theorem]{Corollary}

\theoremstyle{definition}
\newtheorem*{remark}{Remark}

\def\llbracket{[\hspace{-.10em} [ }
\def\rrbracket{ ] \hspace{-.10em}]}
\def\f{\mathcal{F}}
\def\ve{\varepsilon}
\def\la{{\longrightarrow}}
\def\build#1_#2^#3{\mathrel{
\mathop{\kern 0pt#1}\limits_{#2}^{#3}}}
\def\wt{\widetilde}

\title{Scaling limits for the peeling process on random maps}

\author{Nicolas Curien and Jean-Fran\c cois Le Gall \\ {\small \textit{Université Paris-Sud}}}
\date{}

\begin{document}
\maketitle
\begin{abstract} We study the scaling limit of the volume and perimeter 
of the discovered regions in the Markovian explorations known as peeling processes
for  infinite random planar maps such as the uniform infinite planar triangulation (UIPT) or quadrangulation (UIPQ). In particular, our results apply to the metric exploration or
peeling by layers algorithm, where the discovered regions are (almost) completed balls, or hulls,
centered at the root vertex. The scaling limits of the perimeter and volume of hulls
can be expressed in terms of the hull process of the Brownian plane studied in our previous work. Other applications include the  metric exploration of the dual graph of our infinite random lattices, and first-passage percolation with exponential edge weights on the dual graph, also known as the Eden model or uniform peeling. 
\end{abstract}





\section{Introduction}
The spatial Markov property of random planar maps is one of the most important properties of these random lattices. Roughly speaking, this property says that, after a region of the map has been explored, the law of the remaining part only depends on the perimeter of the discovered region. The spatial Markov property was  first used in the physics literature, without a precise justification: Watabiki \cite{Wat95} introduced the so-called ``peeling process'', which is a growth process
discovering the random lattice step by step. A rigorous version of the peeling
process and its Markovian properties
was given by Angel \cite{Ang03} in the case of the Uniform Infinite Planar Triangulation (UIPT), which had been defined by Angel and Schramm \cite{AS03} as the local limit of 
uniformly distributed plane triangulations with a fixed size. The peeling process has been used since  to derive information about the metric 
properties of the UIPT \cite{Ang03}, 
about percolation \cite{Ang03,ACpercopeel,MN13} and simple random walk \cite{BCsubdiffusive} on the UIPT and its generalizations, and more recently about the conformal structure \cite{CurKPZ} of random planar maps. It also plays a crucial role in the construction of ``hyperbolic'' random triangulations \cite{AR13,CurPSHIT}. 

In the present paper, we derive scaling limits for the perimeter and the volume 
of the discovered region in a peeling process of the UIPT. Our methods also apply to
the Uniform Infinite Planar Quadrangulation (UIPQ), which was constructed independently by Krikun \cite{Kri05} and by
Chassaing and Durhuus \cite{CD06} (the equivalence between these two consructions was obtained by M\'enard \cite{Men08}). By considering the
special case of the peeling by layers, we get scaling limits for  the volume and 
the boundary length  of the hull of radius $r$ centered at the root of the UIPT, or of the UIPQ (the hull of radius $r$ is obtained by ``filling in the finite holes'' in the ball of radius $r$).
The limiting processes that arise in these scaling limits coincide with those that
appeared in our previous 
work \cite{CLGHull} dealing with the hull process of the Brownian plane. This is not surprising since the Brownian plane is conjectured to be the universal scaling limit
of many infinite random lattices such as the UIPT, and it is known that this conjecture holds in the special case of the UIPQ \cite{CLGplane}. 
We also apply our results to both the dual graph distance and the first-passage percolation distance 
corresponding to exponential edge weights on the dual graph of the UIPT (this
first-passage percolation model is also known as the Eden model). In particular, we show that the volume and perimeter of the hulls
with respect to each of these two metrics have the same scaling limits 
as those corresponding to the graph distance, up to explicit deterministic multiplicative factors. \bigskip

For the sake of clarity, the following results are stated and proved in the case of the UIPT
corresponding to type II triangulations in the terminology of Angel and Schramm \cite{AS03}. In type II triangulations, loops are not allowed but there may be multiple edges.
Section \ref{sec:general} explains the changes that are needed for the extension
of our results to other random lattices such as the UIPT 
for type I triangulations or the UIPQ. In these extensions, scaling limits remain the same, but different constants are involved.
In the case of type II triangulations, the three basic constants that
arise in our results are
$$ \mathsf{p}_{\two} = (\frac{2}{3})^{2/3}, \quad \mathsf{v}_{\two} = (\frac{2}{3})^{7/3} \quad \mbox{and } \quad  \mathsf{h}_{\two} = 12^{-1/3}\;.$$
Here the subscript $\two$ emphasizes the fact that these constants are relevant
to the case of type II triangulations.

So, except in Section \ref{sec:general}, all triangulations in this article are type II triangulations. The corresponding UIPT is denoted by $T_{\infty}$. This is an infinite random triangulation of the plane given with a distinguished oriented edge whose tail vertex is called the origin (or root vertex) of the map.  If $\mathbf{t}$ is a  rooted finite triangulation with a simple boundary $\partial \mathbf{t}$, we denote the number of inner vertices of $\mathbf{t}$ by $|\mathbf{t}|$ and the boundary length of
$\mathbf{t}$ by $| \partial \mathbf{t}|$. Furthermore,
we say that $\mathbf{t}$ is a subtriangulation of $T_{\infty}$ and write $\mathbf{t} \subset T_{\infty}$, if $T_{\infty}$ is obtained from $\mathbf{t}$ by gluing an infinite triangulation with a simple boundary along the boundary of $\mathbf{t}$ (of course we 
also require that the root of $T_\infty$ coincides with the root of $\mathbf{t}$
after this gluing operation).  If $\mathbf{t} \subset T_{\infty}$ and $e$ is an edge of $\partial \mathbf{t}$, the triangulation obtained by the peeling of $e$ is the triangulation $\mathbf{t}$ to which we add the face incident to $e$ that was not
already in $\mathbf{t}$, as well as the finite region that the union of 
$\mathbf{t}$ and this added face may enclose (recall that the UIPT 
has only one end \cite{AS03}). An exploration process $  (\mathsf{T}_{i})_{i\geq 0}$ is a sequence of subtriangulations of the UIPT with a simple boundary such that $ \mathsf{T}_{0}$ consists only of the root edge
(viewed as a trivial triangulation) and for every $i \geq 0$ the map $ \mathsf{T}_{i+1}$ is obtained from $ \mathsf{T}_{i}$ by peeling one edge of its boundary. If the choice of this edge is independent of $ T_{\infty} \backslash \mathsf{T}_{i}$, the exploration is
said to be \emph{Markovian} and we call it a  \emph{peeling process}. Different peeling processes correspond to different ways of choosing the edge to be peeled at
every step. 
See Section \ref{sec:peeling} for a more rigorous presentation.

Our first theorem complements results due to Angel \cite{Ang03} by describing the scaling limit of the perimeter and volume of the discovered region in a peeling process. We let  $(S_{t})_{t \geq 0}$ denote the stable Lévy process with index $3/2$ and only negative jumps,  which starts from $0$ and is normalized so that its Lévy measure is $3/(4 \sqrt{\pi}) |x|^{-5/2} \mathbf{1}_{x < 0}$, or equivalently $ \mathbb{E}[\exp( \lambda S_{t})] = \exp( t \lambda^{3/2})$ for any $\lambda,t \geq 0$. 
The process $(S_{t})_{t \geq 0}$ conditioned to stay nonnegative
is then denoted by $(S_{t}^+)_{t \geq 0}$   (see \cite[Chapter VII]{Ber96} for 
a rigorous definition of $(S_{t}^+)_{t \geq 0}$). We also let $\xi_{1}, \xi_{2}, \ldots $ be a sequence of  independent real random variables with density  
$$ \frac{1}{\sqrt{2\pi x^5}} e^{- \frac{1}{2x}} \mathbf{1}_{\{x >0\}}\,.$$ 
We assume that this sequence is independent of the process $(S_{t}^+)_{t \geq 0}$
and, for every $t \geq 0$, we set $Z_{t} = \sum_{t_{i} \leq t} \xi_{i} \cdot (\Delta S_{t_{i}}^+)^2$ where $t_{1}, t_{2}, \ldots$ is a measurable enumeration of the jumps of $S^+$.

\begin{theorem}[Scaling limit for general peelings] \label{thm:scalingpeeling} For any peeling process $( \mathsf{T}_{n})_{n \geq 0}$ of the UIPT, we have the following convergence in distribution in the sense of Skorokhod
$$ \left( \frac{| \partial \mathsf{T}_{[nt]} |}{ \mathsf{p}_{\two} \cdot n^{2/3}}, \frac{| \mathsf{T}_{[nt]}|}{ \mathsf{v}_{\two} \cdot n^{4/3}} \right)_{t \geq 0}  \xrightarrow[n\to\infty]{(d)}  \left(S^+_{t},Z_{t}\right)_{t \geq 0}.$$  
\end{theorem}

The proof of Theorem \ref{thm:scalingpeeling} relies on the 
explicit expression of the transition probabilities of the peeling process. It
follows from this explicit expression that 
the process of perimeters $(| \partial \mathsf{T}_{n} |)_{n\geq 0}$ is 
a $h$-transform of a random walk with independent increments in the domain of attraction of a spectrally negative stable distribution with index $3/2$
(Proposition \ref{prop:positive}). This $h$-transform is interpreted
as   conditioning the random walk  to stay above level $2$, and in the scaling
limit this leads to the process  $(S_{t}^+)_{t \geq 0}$. The common 
distribution of the variables $\xi_{i}$ is the scaling limit of the volume of a  
Boltzmann triangulation (see Section \ref{sec:enumer})
conditioned to have a large boundary size. The appearance of this distribution
is explained by the fact that the ``holes'' created by the peeling process
are filled in by finite triangulations distributed according to Boltzmann weights
(this is called the free distribution in \cite[Definition 2.3]{AS03}). As a corollary of Theorem \ref{thm:scalingpeeling}, we prove that any peeling process of the UIPT will eventually discover the whole triangulation, i.e, $\cup \mathsf{T}_{n} = T_{\infty}$, no matter what peeling algorithm is used (of course as long as the exploration is Markovian), see Corollary \ref{cor:bouffetout}. 
We note that Theorem \ref{thm:scalingpeeling} can be applied to
various peeling processes that have been considered in earlier works:  peeling along percolation interfaces \cite{Ang03,ACpercopeel}, peeling along simple random walk \cite{BCsubdiffusive}, peeling along a Brownian or a SLE$_{6}$ exploration of the Riemann surface 
associated with the UIPT \cite{CurKPZ}, etc.  In the present work, we apply Theorem \ref{thm:scalingpeeling} to three specific peeling algorithms, each of which
is related to a  ``metric'' exploration of the UIPT. The first one is the peeling by layers, which essentially grows balls for the graph distance on the UIPT, the second one is the peeling by layers in the dual map of the UIPT and the last one is the uniform peeling, which is related to first-passage percolation with exponential edge weights on
the dual map of the UIPT.  \medskip

\paragraph{Scaling limits for the hulls.}  For every integer $r\geq 1$, 
the ball $B_{r}(T_{\infty})$ is defined as the union of all faces of $T_\infty$
whose boundary contains at least one vertex at graph distance smaller than or equal to $r-1$ from the origin (when $r=0$ we agree that  $B_{0}(T_{\infty})$ is the trivial
triangulation consisting only of the root edge). 
The hull $B_{r}^\bullet(T_{\infty})$  is then obtained by adding to the ball $B_{r}(T_{\infty})$ the
bounded components of the complement of this ball
(see Fig.\,\ref{fig:boules}). Note that $B_{r}^\bullet(T_{\infty})$ is a finite
triangulation with a simple boundary. One can define
a particular peeling process $  (\mathsf{T}_{i})_{i\geq 0}$ (called the peeling by
layers) such that, for every $n\geq 0$, there exists a random integer $H_n$
such that $B_{H_n}^\bullet(T_{\infty})\subset\mathsf{T}_n\subset B_{H_n+1}^\bullet(T_{\infty})$. Scaling limits for the volume and the boundary 
length of the hulls can then be derived by applying Theorem \ref{thm:scalingpeeling}
to this particular peeling algorithm. A crucial step in this derivation is to get information about the
asymptotic behavior of $H_n$ when $n\to\infty$ (Proposition \ref{prop:scalinglayers}). Before stating our limit theorem for hulls, we need
to introduce some notation.

\begin{figure}[h]
\begin{center}\includegraphics[width=13cm]{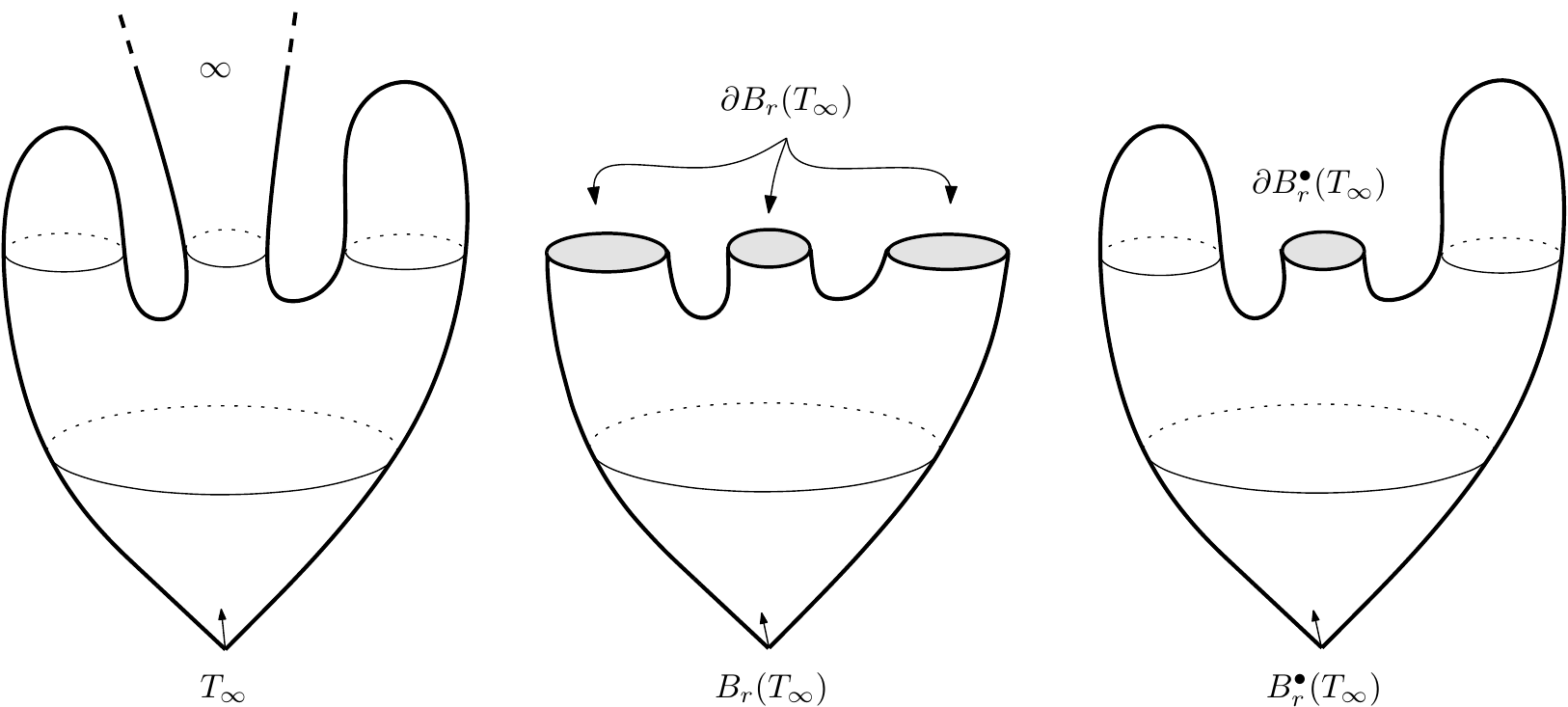}
\caption{ \label{fig:boules} From left to right, the ``cactus'' representation of the UIPT, the ball $B_r(T_{\infty})$, whose boundary may have several components, and  the hull
$B^\bullet_{r}(T_{\infty})$,
whose boundary  is a simple cycle.}
\end{center}
 \end{figure}
 
 For every real $u\geq 0$, set $\psi(u)=u^{3/2}$. The continuous-state branching process with branching mechanism $\psi$ is
the Feller Markov process $(X_t)_{t\geq 0}$ with values in $ \mathbb{R}_+$, whose semigroup is characterized as
follows: for
every $x,t\geq 0$ and every $\lambda> 0$,
$$E[e^{-\lambda X_t}\mid X_0=x]= \exp\Big(-x\Big(\lambda^{-1/2} + t/2\Big)^{-2}\Big).$$
Note that $X$ gets absorbed at $0$ in finite time. It is easy to construct a process
$( \mathcal{L}_t)_{t\geq 0}$ with c\`adl\`ag paths such that 
the time-reversed process $( \mathcal{L}_{(-t)-})_{t \leq 0}$ (indexed by negative times) is distributed
as  $X$ ``started from $+\infty$ at time $-\infty$'' and conditioned to hit zero at time $0$
(see \cite[Section 2.1]{CLGHull} for a detailed presentation of the process $\mathcal{L}$). We consider the sequence $(\xi_{i})_{i\geq 1}$ introduced before Theorem \ref{thm:scalingpeeling}, and
we assume that this sequence is independent of $ \mathcal{L}$. We then set,
for every $t \geq 0$
$$ \mathcal{M}_{t} = \sum_{s_{i} \leq t} \xi_i \cdot ( \Delta \mathcal{L}_{s_{i}})^2,$$ 
where $s_{1}, s_{2},\ldots $ is a measurable enumeration of the jumps of $ \mathcal{L}$.

\begin{theorem}[Scaling limit of the hull process]
 \label{thm:scalinghull} We have the following convergence in distribution in the sense of Skorokhod,
 \begin{eqnarray*} \left( n^{-2}| \partial B^\bullet_{[nt]}(T_\infty)|, n^{-4}|  B^\bullet_{[nt]}(T_\infty)|\right)_{t \geq 0} & \xrightarrow[n\to\infty]{(d)} & \Big(\mathsf{p}_{\two}\cdot \mathcal{L}_{t/ \mathsf{h}_{\two}}, \mathsf{v}_{\two} \cdot \mathcal{M}_{t/ \mathsf{h}_{\two}}\Big)_{t\geq 0}.  \end{eqnarray*}
\end{theorem}

A scaling argument shows that the limiting process has the same distribution as
$$\Big(\frac{\mathsf{p}_{\two}}{(\mathsf{h}_\two)^2}\, \mathcal{L}_{t}, 
\frac{\mathsf{v}_{\two}}{( \mathsf{h}_{\two})^4}\, \mathcal{M}_{t}\Big)_{t\geq 0}$$
but the form given in Theorem \ref{thm:scalinghull} helps to understand
the connection with Theorem \ref{thm:scalingpeeling}. 

We note that the convergence in distribution of the variables $r^{-2}| \partial B^\bullet_{r}(T_\infty)|$
as $r\to\infty$
had already been obtained 
by Krikun \cite[Theorem 1.4]{Kri04} via a different approach. 
The limiting process in Theorem \ref{thm:scalinghull} appeared
in the companion paper \cite{CLGHull} as the process describing the evolution of the
boundary length and the volume of hulls in the Brownian plane
(in the setting of the Brownian plane, the length of the boundary has to
be defined in a generalized sense). The paper
\cite{CLGHull}
contains detailed information about distributional properties of this 
limiting process (see Proposition 1.2 and Theorem 1.4
in \cite{CLGHull}). In particular, for every fixed $s>0$, the joint distribution of
the pair $(\mathcal{L}_s,\mathcal{M}_s)$ is known explicitly. Here we mention
only the Laplace transform of the marginal laws:
 \begin{eqnarray*} \mathbb{E}[e^{-\lambda \mathcal{L}_s}] &=& (1 + \frac{\lambda s^2}{4})^{-3/2},\\
\mathbb{E}[e^{-\lambda\mathcal{M}_s}] &=& 3^{3/2}\,\cosh\Big(\frac{(2\lambda)^{1/4}s}{
\sqrt{8/3}}\Big)
\,\Big(\cosh^2\Big(\frac{(2\lambda)^{1/4}s}{\sqrt{8/3}}\Big) + 2\Big)^{-3/2} .  \end{eqnarray*}
Note in particular that $\mathcal{L}_r$ follows a Gamma distribution 
with parameter $3/2$.

\paragraph{Metric exploration of the dual map.} Consider now the 
dual map $T^*_{\infty}$ of the UIPT, whose vertices are the faces of the UIPT, and
where two vertices
are connected by an edge if the corresponding faces of the UIPT share a common
edge. The origin of $T_{\infty}^*$, or root face of $T_{\infty}$, is the face incident to the right side of the root edge of $T_{\infty}$. We 
equip $T^*_{\infty}$ with the dual graph distance, and we let $B_{r}^{\bullet, *}(T_{\infty})$ denote the hull of the ball of radius $r$ in $T^*_{\infty}$, i.e.~the map made of all the faces of $T_{\infty}$ which are at dual graph distance less than or equal to $r$ from the root face, together with the finite regions these faces may enclose. Then the techniques developed for the proof of Theorem \ref{thm:scalinghull} also give the following result.

\begin{theorem}[Scaling limit of the hull process on the dual map]
 \label{thm:scalinghulldual} We have the following convergence in distribution in the sense of Skorokhod,
 \begin{eqnarray*} \left( n^{-2}| \partial B^{\bullet,*}_{[nt]}(T_\infty)|, n^{-4}|  B^{\bullet,*}_{[nt]}(T_\infty)|\right)_{t \geq 0} & \xrightarrow[n\to\infty]{(d)} & \Big(\mathsf{p}_{\two}\cdot \mathcal{L}_{t/ \mathsf{h}^*_{\two}}, \mathsf{v}_{\two} \cdot \mathcal{M}_{t/ \mathsf{h}^*_{\two}}\Big)_{t\geq 0},  \end{eqnarray*}
 where $ \mathsf{h}^*_{\two} = \mathsf{h}_{\two} + \big(  \mathsf{p}_{\two}\big)^{-1}$.
\end{theorem}

\paragraph{First-passage percolation.} Consider again the dual map $T_{\infty}^*$ of the UIPT. We assign independently to each edge of the dual map an exponential
weight with parameter $1$. For every $t\geq 0$,  
we write $ \mathsf{F}_{t}$ for the union of all faces that may be reached from the root face by a (dual) path 
whose total weight is at most $t$. As usual, $ { \mathsf{F}}^\bullet_{t}$ stands for the hull of $ \mathsf{F}_{t}$, which is obtained by filling in the finite holes of $ \mathsf{F}_{t}$ inside $ T_{\infty}$. Then $  \mathsf{F}^\bullet_{t}$ is a triangulation with a simple boundary. If $ 0 = \tau_{0} < \tau_{1} < \cdots < \tau_{n} \ldots$ are the jump times of the process $t\mapsto  \mathsf{ F}^\bullet_t$, it
is not hard to verify that the sequence $ (\mathsf{F}^\bullet_{\tau_{n}})_{n\geq 0}$ is a \emph{uniform} peeling process, meaning that at each step
the edge to be peeled off is chosen uniformly at random among
all edges of the boundary. See Proposition  \ref{prop:fpppeeling} for a precise
statement. Then Theorem \ref{thm:scalingpeeling} leads to the following result:

\begin{figure}[!h]
 \begin{center}
 \includegraphics[height=6cm]{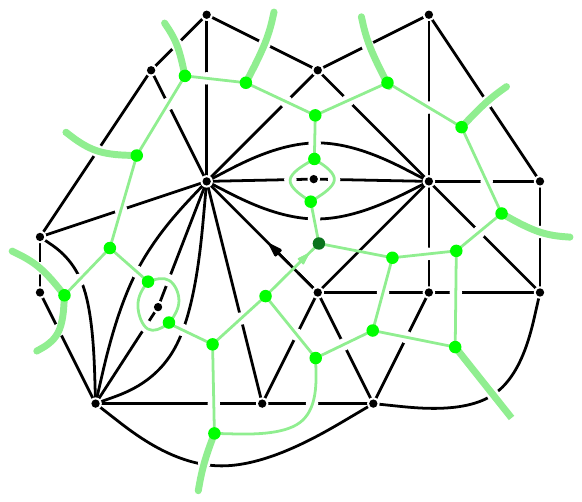}
 \caption{Illustration of the exploration along first-passage percolation on the dual of the UIPT. We represented $ \mathsf{F}^\bullet_{t}$ for some value of $t>0$. By standard properties of exponential variables, the next dual edge to be explored is uniformly distributed on the boundary.}
 \end{center}
 \end{figure}

\begin{theorem}[Scaling limits for first passage percolation]  \label{thm:scalingfpp} We have the following convergence in distribution for the Skorokhod topology 
 \begin{eqnarray*} \left(  n^{-2}|\partial \mathsf{F}^\bullet_{[nt]}|, n^{-4}| \mathsf{F}^\bullet_{[nt]}|\right)_{t \geq 0}  &\xrightarrow[n\to\infty]{(d)}&  \Big(\mathsf{p}_\two\cdot\mathcal{L}_{\mathsf{p}_\two  t   }, \mathsf{v}_\two\cdot\mathcal{M}_{\mathsf{p}_\two t }\Big)_{t \geq 0}.  \end{eqnarray*}
\end{theorem}

Set $c_{1} =  \mathsf{h}^*_{\two}/ \mathsf{h}_{\two} = 4$ and $c_{2}= (\mathsf{p}_\two \mathsf{h}_\two)^{-1}= 3$. If we compare Theorem \ref{thm:scalinghull}, Theorem \ref{thm:scalinghulldual} and Theorem \ref{thm:scalingfpp},
we see that the scaling limits of the volume and
the perimeter are the same for  ${B}^\bullet_{r}(T_\infty)$, for ${B}^{\bullet,*}_{ c_{1} \cdot r}(T_\infty)$ and for $\mathsf{F}^\bullet_{c_{2}\cdot r}$. This is consistent with the conjecture 
saying
that balls for  the dual graph distance or for first-passage percolation distance grow like deterministic balls,
up to a constant multiplicative factor (this property is not expected to hold for
deterministic lattices such as $\mathbb{Z}^2$, but in some sense the UIPT is more isotropic).
Informally, writing $\mathrm{d_{gr}}$ for the graph distance (on the UIPT), $ \mathrm{d}_{ \mathrm{gr}}^*$ 
 for the dual graph distance and $\mathrm{d_{fpp}}$ for the first-passage percolation distance, our results
suggest that in large scales,
$$  \mathrm{d}_{ \mathrm{gr}}^*(\cdot,\cdot)  \approx  {c_{1}} \cdot \mathrm{d_{gr}}(\cdot, \cdot)\qquad \mathrm{d_{fpp}}(\cdot,\cdot)  \approx  {c_{2}} \cdot \mathrm{d_{gr}}(\cdot, \cdot).$$
Note that $ \mathrm{d_{gr}}$ is a metric on the UIPT, whereas $ \mathrm{d_{fpp}}$ or $ \mathrm{d}_{ \mathrm{gr}}^*$ are metrics on the dual graph. Still it is easy to restate the previous
display in the form of a precise conjecture (see Section \ref{rem:fpp}). This conjecture is
consistent with the recent calculations of Ambj\o rn and Budd \cite{AB14} for two and three-point functions in first-passage percolation on random triangulations, and is the subject of the forthcoming work \cite{CLGmodif}.

We finally note that our uniform peeling process can be viewed as a variant of
the classical Eden model on the (dual graph of the) UIPT. The same variant 
has been considered by Miller and Sheffield \cite{MS13}
and served as a motivation for the construction of Quantum Loewner Evolutions.
In fact the process  QLE$(\frac{8}{3},0)$ that is constructed in \cite{MS13}
is a continuum analog of the Eden model on the UIPT. See Section 2.2 in
\cite{MS13} for more details. 

The organization of the paper follows the preceding presentation. In Section 2,
we recall some enumeration results for triangulations that play an important
role in the paper, and we also give a result connecting the UIPT with Boltzmann
triangulations, which is of independent interest (Theorem \ref{theo:martingale}).
This result shows that the distributions of the ball of radius $r$ in the UIPT and in a
Boltzmann triangulation are linked by an absolute continuity relation 
involving a martingale, which has an explicit expression in terms of the
sizes of the cycles bounding the connected components of the complement
of the ball. 

\smallskip
\noindent{\bf Acknowledgement.} We thank Timothy Budd 
for interesting discussions and for providing Figure \ref{fig:Budd}. We also thank two anonymous referees for several
useful comments.

\tableofcontents	

\section{Preliminaries}

Throughout this work, we consider only {\em rooted} planar maps, and we 
often omit the word rooted. We view planar maps as graphs drawn on the sphere,
with the usual identification modulo orientation-preserving homeomorphisms. 
Recall that, except in Section 6 below,  we restrict our attention to type II triangulations, meaning that there are no loops, but multiple 
edges are allowed. We define a triangulation with a boundary as a rooted
planar map without loops, 
with a distinguished face (the external face)  bounded by a simple cycle
(called the boundary), such that all faces except possibly the distinguished one are triangles.
 If $\tau$ is a triangulation with a boundary, we denote its boundary  by $ \partial \tau$. Vertices of $\tau$ not on the boundary are called inner vertices. 
The size $|\tau|$ of $\tau$ is defined as the number of inner vertices of $\tau$.
The length $ |\partial  \tau|$ of 
$ \partial \tau$ (or perimeter of $\tau$) is the number of edges, or equivalently the number of vertices, in $ \partial \tau$. Note that
$|\partial \tau | \geq 2$ since loops are not allowed.

\subsection{Enumeration}
\label{sec:enumer}

We gather here several results about the asymptotic enumeration of planar triangulations, see \cite{Ang03,AS03} and the references therein.
For every $n\geq 0$ and $p \geq2$, we let $ \mathcal{T}_{n,p}$ denote the set of all (type II) triangulations of size $n$ with a simple boundary  of length $p$, that are rooted at an edge of the boundary 
oriented so that the external face lies on the right of the root edge. We have
\begin{equation}
\label{eq:tnpexact}
   \# \mathcal{T}_{n,p} = \frac{2^{n+1}(2p-3)!(2p+3n-4)!}{(p-2)!^2n!(2p+2n-2)!} \underset{n\to\infty}{\sim} C(p)\big(\frac{27}{2}\big)^n n^{-5/2}
   \end{equation}
where 
   \begin{equation} C(p) = \frac{ 4}{ 3^{7/2}\sqrt{\pi}}  \frac{  (2p-3)!}{(p-2)!^2} \left(\frac{9}{4}\right)^p \underset{p\to \infty}{\sim}   \frac{1}{54 \pi \sqrt{3}} 9^p \sqrt{p}.\label{equivalentcp} 
   \end{equation}
   The exact formula for $ \# \mathcal{T}_{n,p} $ in \eqref{eq:tnpexact} gives $\# \mathcal{T}_{n,p}=1$  for $n=0$ and $p=2$. This
   formula is valid provided we make
  the special 
   convention that the rooted planar map consisting of a single (oriented) edge between two vertices is 
   viewed as a triangulation with a simple boundary of length $2$: This will be called the trivial triangulation. It will be used in
   the sequel as the starting point of the peeling process, and also sometimes to ``fill in'' holes of size two arising in this process. 
   
The exponent $5/2$ in \eqref{eq:tnpexact} is typical of the enumeration of planar maps and shows that 
$$Z(p):=\sum_{n=0}^\infty \Big(\frac{2}{27}\Big)^n\, \#  \mathcal{T}_{n,p}<\infty.$$ 
The numbers $Z(p)$ can be computed 
exactly (see \cite[Proposition 1.7]{Ang03}): for every $p\geq 2$,
  \begin{equation}
  \label{eq:defzp}
  Z(p) \;=\; \frac{(2p-4)!}{(p-2)!p!}\left( \frac{9}{4}\right)^{p-1}. \end{equation}
  
 Triangulations in $ \mathcal{T}_{n,p}$, for some $n\geq 0$,  are also called triangulations of the $p$-gon.
By definition, the (critical) Boltzmann distribution on triangulations of the $p$-gon is the probability
  measure on $\bigcup_{n\geq 0} \mathcal{T}_{n,p}$ that assigns mass
  $(2/27)^{n} Z({p})^{-1}$ to each triangulation of $ \mathcal{T}_{n,p}$.
  This is also called the free distribution in \cite{AS03}. It follows from \eqref{eq:defzp} that for every $x \in [0,1/9]$,
  $$ \sum_{p=1}^\infty Z(p+1) \,x^p =  \frac{1}{2} + \frac{(1-9x)^{3/2}-1}{27x}.$$
From \eqref{eq:defzp}  and the last display, we get that 
\begin{eqnarray} Z(p+1) \!\!&\underset{p\to \infty}{\sim}& \!\!     \mathsf{t}_{ \two} \cdot  9^p p^{-5/2}, \quad \mbox{ where }\  \mathsf{t}_{\two} = \frac{1}{4 \sqrt{\pi}} ,\label{eq:asympZp}\\ 
\sum_{p=1}^\infty Z(p+1)\,9^{-p}  \!\! \!\!&=& \!\!  \!\!\frac{1}{6} \label{eq:sumZp}\;,\\
\sum_{p=1}^\infty p\, Z(p+1)\,9^{-p} \!\!  \!\!&=& \!\!  \!\! \frac{1}{3} \label{eq:sumkZp}\;.  \end{eqnarray}

Finally, we note that there is a bijection between rooted triangulations of the $2$-gon having $n$ inner vertices and rooted plane triangulations having $n+2$ vertices: Just glue together the two boundary edges of a triangulation of the $2$-gon to get a
triangulation of the sphere. The Boltzmann distribution on rooted triangulations of the $2$-gon thus induces a 
probability measure on the space of all  triangulations of the sphere (including the trivial one). A random triangulation
distributed according to this probability measure is called a Boltzmann triangulation of the sphere. 
Equivalently, the law of a Boltzmann triangulation of the sphere assigns a mass 
 $(2/27)^{n-2} Z(2)^{-1}$ to every triangulation of the sphere with $n$ vertices (including the trivial
triangulation for which $n=2$). 

\subsection{Boltzmann triangulations and the UIPT}

In this section, we describe a relation between Boltzmann triangulations of the sphere and the UIPT. This relation is 
not really needed in what follows but it helps to understand the importance of Boltzmann triangulations in the subsequent developments. 

Let $T_{\rm Bol}$ be a Boltzmann triangulation of the sphere. As 
in the introduction above, for every integer $r\geq 1$, let
$B_{r}(T_{\rm Bol})$ denotes the ball of radius $r$ in $T_{\rm Bol}$. 
So $B_{r}(T_{\rm Bol})$ is the rooted planar map obtained by
keeping only those faces of $T_{\rm Bol}$ that are incident to at least
one vertex at distance at most $r-1$ from the root vertex.
We view $B_{r}(T_{\rm Bol})$ as a random variable with values in the
space of all (type II) triangulations with holes. Here, a triangulation
with holes is a planar map without loops, with a finite number of distinguished faces called the holes, 
such that all faces except possibly the holes are triangles, the boundary of every hole is a simple cycle, whose length is called the
size of the hole, and two
distinct holes cannot share a common edge
(the triangulations with a simple boundary that we considered above are just triangulations with a single hole). 
In the case of $B_{r}(T_{\rm Bol})$, holes obviously correspond to the connected components of the complement
of the ball, in a way analogous to the middle part
of Fig.~\ref{fig:boules}.
We write $\ell_{1}(r), \ell_{2}(r), \ldots,
\ell_{n_r}(r)$ for the sizes of the holes of  $B_{r}(T_{\rm Bol})$ enumerated in 
nonincreasing order. We also write $ \mathcal{F}_{r}$ for
the $\sigma$-field  generated by $B_{r}(T_{\rm Bol})$ and we let $ \mathcal{F}_{0}$ be the trivial $\sigma$-field.
Recall our notation $B_r(T_\infty)$ for the ball of radius $r$ in the UIPT, which is also viewed as
a random triangulation with holes. 

\begin{theorem} \label{theo:martingale} Let $f(n) =  \frac{n}{2}\cdot (n-1)\cdot (2n-3)$ for every integer $n\geq 3$ and $f(2)=9$. The random process $(\mathbf{M}_{r})_{r\geq 0}$
defined by
$$  \mathbf{M}_{r}:=\sum_{i=1}^{n_r} f( \ell_{i}(r))\,,\quad \mbox{ for } r \geq 1\,,$$
and $\mathbf{M}_0=1$, 
is a martingale with respect to the filtration $( \mathcal{F}_{r})_{r\geq 0}$.
Moreover, if $F$ is any nonnegative measurable function on the
space of triangulations with holes, we have, for every $r\geq 1$,
\begin{equation}
\label{absocont}
 \mathbb{E}[F(B_{r}(T_\infty))] = \mathbb{E}[\mathbf{M}_r\,F(
B_{r}( {T_{\rm Bol}}))].
\end{equation}
\end{theorem}

The second part of the 
theorem shows that the law of a ball in the UIPT can be obtained by biasing
the law of the corresponding ball in a Boltzmann triangulation using
the martingale $\mathbf{M}_r$. This is an analog of a classical result
for Galton--Watson trees: In order to get the first 
$k$ generations of a Galton--Watson tree conditioned 
on non-extinction, one  biases the law of the first $k$ generations
of an unconditioned Galton--Watson tree using a martingale which is
simply the size of generation $k$ of the tree (see e.g. \cite[Chapter 12]{LP10}).
In a sense, the UIPT can thus be viewed as a Boltzmann triangulation
conditioned to be infinite. This is related to the discussion in Section 6 of \cite{AS03},
which associates with a Boltzmann triangulation a multitype Galton--Watson tree 
describing the structure of balls, in such a way that the tree associated with the UIPT 
is just the same Galton--Watson tree conditioned on non-extinction. 

\proof It suffices to prove the second part of the theorem. Indeed, if \eqref{absocont} holds, 
we immediately get, for every $1\leq k\leq \ell$, and every function $F$,
$$ \mathbb{E}[\mathbf{M}_\ell\,F(
B_{k}( {T_{\rm Bol}}))]=  \mathbb{E}[\mathbf{M}_k\,F(
B_{k}( {T_{\rm Bol}}))],$$
and it follows that $\mathbb{E}[\mathbf{M}_\ell\mid  \mathcal{F}_{k}] = \mathbf{M}_k$.

In order to verify the second assertion of the theorem, we will provide explicit
formulas for the probability that the ball of radius $r$ in $T_{\rm Bol}$, 
resp. in $T_\infty$, is equal to a given triangulation with holes. Let $\mathbf{t}$ be a 
fixed triangulation with holes. Note that $\mathbb{P}(B_{r}( {T_{\rm Bol}})=\mathbf{t})>0$ if and only if all vertices belonging to the
boundaries of the holes of $\mathbf{t}$ are at distance $r$ from the root vertex, and all faces of $\mathbf{t}$ other than the holes are incident
to (at least) one vertex at distance at most $r-1$ from the root vertex. Furthermore,
the preceding conditions are also necessary for
$\mathbb{P}(B_{r}( {T_{\infty}})=\mathbf{t})$ to be positive. 

Write $n$ for the total number
of vertices of $\mathbf{t}$, $m\geq 0$ for the number of holes of $\mathbf{t}$ and $p_1,\ldots,p_m$
for the respective sizes of the holes of $\mathbf{t}$ -- the holes are enumerated in
some deterministic manner given $\mathbf{t}$. Then, for every integer $q\geq n$, the number
of triangulations with $q$ vertices whose ball of radius $r$
coincides with $\mathbf{t}$ is equal to
$$\sum_{n_1+\cdots+n_m=q-n} \Big( \prod_{j=1}^m \#\mathcal{T}_{n_j,p_j}\Big),$$
where the sum is over all choices of the nonnegative integers $n_1,\ldots,n_m$ such that
${n_1+\cdots+n_m=q-n}$, with the additional constraint that $n_i>0$ if $p_i=2$. The reason for this last constraint if
the fact that a hole of size $2$ cannot be filled by the trivial triangulation, because this would mean that
we glue the two edges of the boundary. Note that when there is no hole
($m=0$) the quantity in the last display should be interpreted as equal to $1$
if $q=n$ and to $0$ otherwise. The total Boltzmann weight
of those triangulations whose ball of radius $r$
coincides with $\mathbf{t}$ is then
$$\sum_{q=n}^\infty \Big(\frac{2}{27}\Big)^{q-2} Z(2)^{-1} \sum_{n_1+\cdots+n_m=q-n} \Big( \prod_{j=1}^m \#\mathcal{T}_{n_j,p_j}\Big),$$
where we impose the same constraint as before on the integers $n_1,\ldots,n_m$ in the sum. 
We set $Z'(p)=Z(p)$ if $p>2$ and $Z'(2)=Z(2)-1$. 
The quantity in the last display equals 
$$
\Big(\frac{2}{27}\Big)^{n-2} Z(2)^{-1}\;\sum_{n_1=\mathbf{1}_{\{p_1=2\}}}^\infty 
\ldots \sum_{n_m=\mathbf{1}_{\{p_m=2\}}}^\infty
\prod_{j=1}^m \Big( \Big(\frac{2}{27}\Big)^{n_j}\,\#\mathcal{T}_{n_j,p_j}\Big)
= \Big(\frac{2}{27}\Big)^{n-2} Z(2)^{-1}\;\prod_{j=1}^m Z'(p_j),
$$
and so we have proved that 
\begin{equation}
\label{lawball-Bol}
\mathbb{P}(B_{r}( {T_{\rm Bol}})=\mathbf{t}) = \Big(\frac{2}{27}\Big)^{n-2} Z(2)^{-1}\;\prod_{j=1}^m Z'(p_j).
\end{equation}

Next consider the UIPT $T_\infty$. We can similarly compute $\mathbb{P}(B_{r}( T_{\infty})=\mathbf{t})$, using the fact that $T_\infty$ is the local limit of triangulations with a large size.
If, for every integer $q\geq 3$, $T_{(q)}$ denotes a 
uniformly distributed plane triangulation with $q$ vertices, we have 
$$\mathbb{P}(B_{r}( T_{\infty})=\mathbf{t}) = \lim_{q\to\infty}
\mathbb{P}(B_{r}( T_{(q)})=\mathbf{t}) .$$
Recalling that the number of rooted plane triangulations (of type II) with $q$ vertices is $\#\mathcal{T}_{q-2,2}$
the same counting argument as above gives for $q\geq n$,
$$\mathbb{P}(B_{r}( T_{(q)})=\mathbf{t}) 
= \Big( \#\mathcal{T}_{q-2,2}\Big)^{-1} \sum_{n_1+\cdots+n_m=q-n}
\Bigg(\prod_{j=1}^m \#\mathcal{T}_{n_j,p_j}\Bigg),$$
where the sum is again over nonnegative integers $n_1,\ldots,n_m$ such that
${n_1+\cdots+n_m=q-n}$, with the same additional constraint that $n_i>0$ if $p_i=2$.
From the asymptotics in \eqref{eq:tnpexact}, it is an easy matter to verify that, for
any $\ve>0$, we can choose $K$ sufficiently large so that the 
asymptotic contribution of terms corresponding to
choices of $n_1,\ldots,n_m$ where $n_i\geq K$ for two distinct values of $i\in\{1,\ldots,m\}$
is bounded above by $\ve$ (compare with \cite[Lemma 2.5]{AS03}). Thanks to this observation, we get from the asymptotics
 \eqref{eq:tnpexact} that
 $$\mathbb{P}(B_{r}( T_{\infty})=\mathbf{t})= \Big(\frac{2}{27}\Big)^{n-2} C(2)^{-1}
 \sum_{j=1}^m C(p_j) \sum_{n_1,\ldots,n_{j-1},n_{j+1},\ldots,n_m} \Bigg(\,\underset{i\not = j}{\prod_{i=1}^m} 
\Big(\frac{2}{27}\Big)^{n_i} \#\mathcal{T}_{n_i,p_i}\Bigg),$$
 where the second sum is over all choices of $n_1,\ldots,n_{j-1},n_{j+1},\ldots,n_m\geq 0$ such that
  $n_i>0$ if $p_i=2$. It follows that
\begin{equation}
\label{lawball-UIPT}
\mathbb{P}(B_{r}( T_{\infty})=\mathbf{t})=
   \Big(\frac{2}{27}\Big)^{n-2} C(2)^{-1}  \sum_{j=1}^m C(p_j)
 \Bigg(  \;\underset{i\not = j}{\prod_{i=1}^m} \,Z'(p_i)\Bigg).
 \end{equation}
 Comparing \eqref{lawball-UIPT} with \eqref{lawball-Bol}, we get
$$\mathbb{P}(B_{r}( {T_\infty})=\mathbf{t}) = \Big(\frac{Z(2)}{C(2)}\,\sum_{j=1}^m \frac{C({p_j})}{Z'(p_j)}\Big)\,\mathbb{P}(B_{r}( {T_{\rm Bol}})=\mathbf{t}).$$
Note that, for every integer $p\geq 2$, 
$$\frac{Z(2)}{C(2)}\,\frac{C(p)}{Z'(p)}=f(p),$$
and so we have obtained 
$\mathbb{P}(B_{r}( {T_\infty})=\mathbf{t}) =g(\mathbf{t})\,\mathbb{P}(B_{r}( {T_{\rm Bol}})=\mathbf{t})$, where $g(\mathbf{t}):= \sum_{j=1}^m f(p_j)$. Formula 
\eqref{absocont} now follows since $\mathbf{M}_r=g(B_{r}( {T_{\rm Bol}}))$ by definition.
\endproof

\begin{remark} Formula \eqref{lawball-UIPT} is obviously related to Proposition 4.10 in \cite{AS03}. 
We did not use directly that result because it is apparently restricted to type III triangulations (the formula of Proposition 4.10 in \cite{AS03}
does not seem to take into account the possibility of holes of size $2$). 
\end{remark}

\section{Asymptotics for a general peeling process}

\subsection{Peeling}

\label{sec:peeling}

The peeling process is an algorithmic procedure that ``discovers''  the UIPT step by step. We give a brief presentation of this algorithm and refer to \cite{Ang05,Ang03,ACpercopeel,BCsubdiffusive} for details. \medskip

 Formally, the algorithm produces a nested sequence of rooted triangulations with a simple boundary $ \mathsf{T}_{0}\subset  \mathsf{T}_{1} \subset \ldots \subset  \mathsf{T}_{n} \subset \ldots \subset T_{\infty}$, such that, for every $i \geq 0$, conditionally on $\mathsf{T}_{i}$, the remaining part $T_{\infty}\backslash  \mathsf{T}_{i}$  has the same distribution as a UIPT of the $| \partial  \mathsf{T}_{i}|$-gon  (see \cite[Section 1.2.2]{Ang03} for the definition
 of the UIPT of the $p$-gon). 
 
Assuming that we are given the UIPT $T_\infty$, the sequence $ \mathsf{T}_{0}, \mathsf{T}_{1},\ldots$ is  constructed inductively as follows. 
 First $\mathsf{T}_0$ is the trivial triangulation. Then, for every $n \geq 0$, conditionally on $ \mathsf{T}_{n}$ we pick
 an edge $e_{n}$ on $ \partial  \mathsf{T}_{n}$,  \emph{either deterministically (i.e. as a deterministic function of $\mathsf{T}_n$) or via a randomized algorithm} 
 that  may involve only random quantities independent of $T_\infty$. The 
 triangulation $\mathsf{T}_{n+1}$ is then obtained by adding to $\mathsf{T}_n$ the triangle incident to $e_n$ which 
 was not contained in $\mathsf{T}_n$ (this is called the revealed triangle) 
 and the bounded region that may be enclosed in the union of $\mathsf{T}_n$ and the revealed triangle. 
 We sometimes say that  $\mathsf{T}_{n+1}$ is obtained from $\mathsf{T}_{n}$ by peeling the edge $e_n$. 
 Notice that, at the first step, there is only one (oriented) edge in the boundary of $\mathsf{T}_0$, but we can choose 
to reveal the triangle on the right or on the left of this oriented edge. 
 
 The point is the fact that the distribution of the whole sequence $ \mathsf{T}_{0}, \mathsf{T}_{1},\ldots$ can be described 
 in a simple way and provides a construction of $T_\infty$ (although this is not obvious, we shall see later that
 $T_\infty$ is the limit of the finite triangulations $\mathsf{T}_n$). Remarkably, the description of the law of   
 $ \mathsf{T}_{0}, \mathsf{T}_{1},\ldots$ is essentially the same independently of the (deterministic or randomized) algorithm that we
 use to choose the peeled edge at step $n$. 
 
In order to describe the conditional law of $\mathsf{T}_{n+1}$ given $\mathsf{T}_n$ and the peeled edge $e_n$, we need
to distinguish several cases.
Suppose that at step $n \geq 0$ the triangulation $ \mathsf{T}_{n}$ has a boundary of length $p$. The revealed triangle at time $n$ may be of several different types (see Fig.\,\ref{cases}): 
\begin{enumerate}
 \item  \textbf{Type $ \mathsf{C}$:} The revealed triangle has a vertex in the ``unknown region''. This occurs with probability 
 \begin{equation}
 \label{probatypeC}
 \mathbb{P}( \mathsf{C} \mid | \partial  \mathsf{T}_{n}| = p)= q_{-1}^{(p)} =  \frac{2}{27}\frac{C(p+1)}{C(p)}.
 \end{equation}
 \item   \textbf{Types $ \mathsf{L}_{k}$ and $ \mathsf{R}_{k}$:} The three vertices of the revealed triangle lie on the boundary of $ \mathsf{T}_{n}$. This triangle thus ``swallows'' a piece of the boundary of $ \partial \mathsf{T}_{n}$ of length $k \in \{1, \ldots , p-2\}$. These events are denoted by $ \mathsf{R}_{k}$ or $ \mathsf{L}_{k}$, depending on whether the edge of the revealed triangle that comes after the peeled edge in clockwise order
 is incident or not to the
 infinite part of the triangulation (see Fig.\,\ref{cases}). These events have a probability   equal to
 \begin{equation}
 \label{probatypeR} \mathbb{P}(  \mathsf{L}_{k} \mid | \partial \mathsf{T}_{n}| = p) = \mathbb{P}(  \mathsf{R}_{k} \mid | \partial \mathsf{T}_{n}| = p)   := q^{(p)}_{k} = Z(k+1)\frac{C(p-k)}{C(p)}.  
  \end{equation}

\begin{figure}[h]
\begin{center}
\includegraphics[width=14cm]{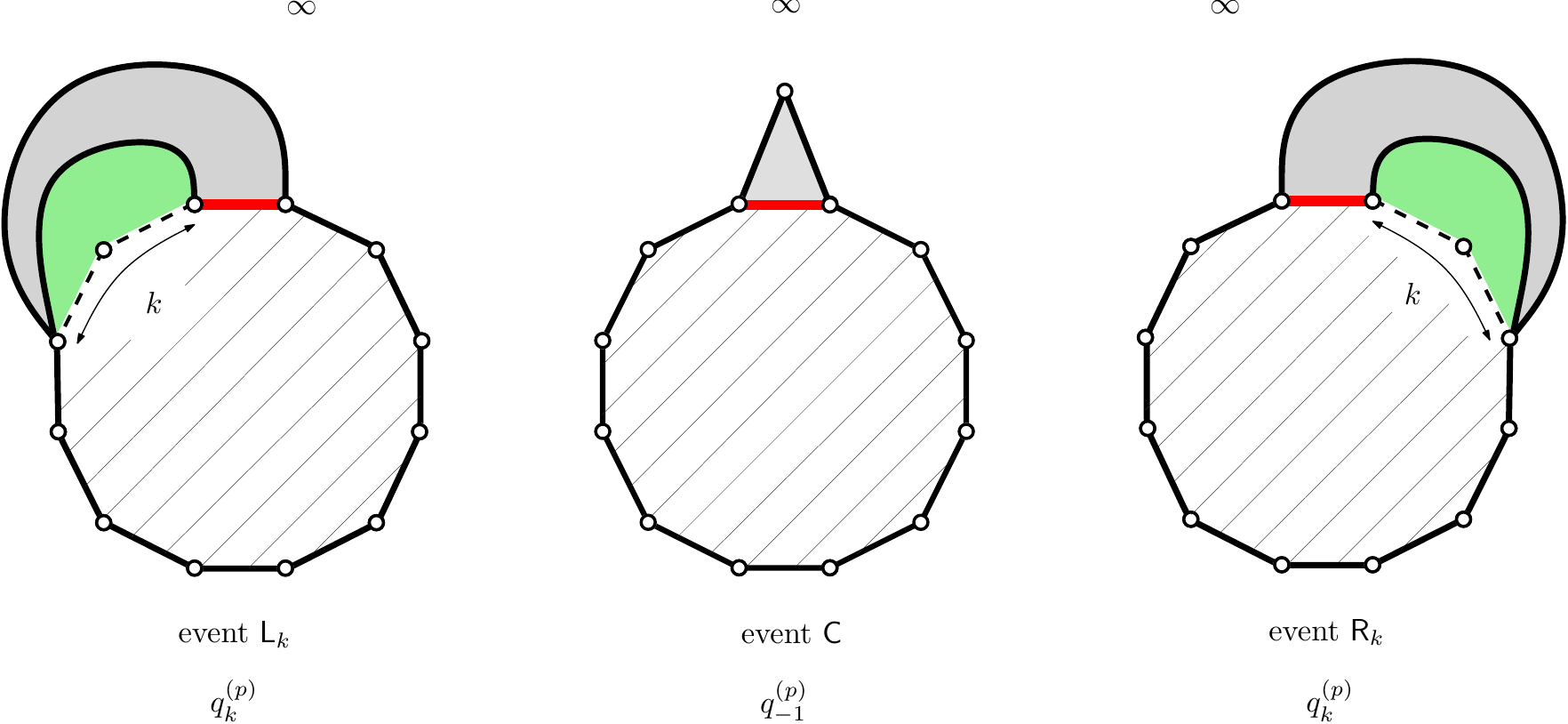}
\end{center}
\caption{ \label{cases}Illustration of cases $ \mathsf{C}$, $ \mathsf{L}_{k}$, and $ \mathsf{R}_{k}$. 
}
\end{figure}
\end{enumerate}

In cases $ \mathsf{R}_{k}$ and $ \mathsf{L}_{k}$, we also need to specify the distribution of the triangulation with a boundary of length $k+1$ that
is enclosed in the union of $\mathsf{T}_k$ and the revealed triangle.  If by convention 
we root this triangulation at the unique edge of its boundary incident to the revealed triangle,
we specify its distribution by saying that it is a Boltzmann triangulation of the $(k+1)$-gon. Note that when $k=1$, there is a positive probability that this Boltzmann triangulation is the trivial one, 
and this simply
means that the enclosed region is empty, or equivalently that the revealed triangle has two edges on the boundary of $\mathsf{T}_n$. 

The preceding considerations completely describe the distribution of the sequence $ \mathsf{T}_{0}, \mathsf{T}_{1},\ldots$ -- modulo
of course the deterministic or randomized algorithm that is used at every step to select the peeled edge. The choices 
of types $ \mathsf{C}$, $ \mathsf{L}_{k}$, and $ \mathsf{R}_{k}$, and of the Boltzmann triangulations that are used (whenever needed)
to ``fill in the holes'' are made independently at every step with the probabilities given above. 

At this point, we note that the geometry of the random triangulations $ \mathsf{T}_{n}$ depends on the peeling algorithm used to choose the peeled edge at every step. On the other hand, it should be clear from the previous description that the law of the process $( | \mathsf{T}_{n}|, |\partial  \mathsf{T}_{n}|)_{n \geq 0}$ does not depend on this algorithm. In the present section, we will be interested only in this process,
and for this reason we do not need to specify the peeling algorithm. Later, in Section 3 and 4, we will consider 
particular choices of the peeling algorithm, which are useful to investigate various properties of the UIPT. 

 To simplify notation, we set, for every $n\geq 0$, 
 $$P_{n} = | \partial  \mathsf{T}_{n}| \quad  \mbox{and} \quad  V_{n} = | \mathsf{T}_{n}|.$$ 

In the remaining part of this section, we will prove Theorem \ref{thm:scalingpeeling} describing the scaling limit of the process $(P_{n},V_{n})_{n \geq 0}$ (see  \cite{Ang03} and \cite[Theorem 5]{BCsubdiffusive} in the quadrangular case for related statements). We will
also establish a few consequences of Theorem \ref{thm:scalingpeeling}, which are  of independent interest.

\subsection{The scaling limit of perimeters}

\label{sec:perimeter}
The description of the previous section shows that both processes $(P_n)_{n\geq 0}$ and $(P_n,V_n)_{n\geq 0}$ are 
Markov chains. The Markov chain
$(P_n)_{n\geq0}$ starts from $P_{0}=2$ and takes values in $\{2,3,\ldots\}$. Its transition probabilities are given by 
 \begin{eqnarray} \label{def:Pn}  \mathbb{E}[f(P_{n+1}) \mid P_{n}] = f(P_{n}+1) \cdot q^{(P_{n})}_{-1} + 2 \sum_{k=1}^{p-2} f(P_{n}-k) \cdot q_{k}^{(P_{n})}.  \end{eqnarray}

Using \eqref{equivalentcp}, we may  set $q_{-1} = \lim_{p \to \infty} q_{-1}^{(p)} = \frac{2}{3}$ and similarly $q_{k} = \lim_{p \to \infty} q_{k}^{(p)} = Z(k+1)9^{-k}$ for every $k \geq 1$. From \eqref{eq:sumZp} and \eqref{eq:sumkZp}, it is an easy matter to
verify that 
$$ q_{-1} + 2 \sum_{k \geq 1}q_{k} =1 \qquad \mbox{ and }\qquad q_{-1} - 2 \sum_{k \geq 1} k \, q_{k} =0,$$
so that the probability measure $\nu$ on $\Z$ given by $\nu(1)=q_{-1}$ and $\nu(-k)=2q_k$ for every $k\geq 1$ is  centered
(note that $\nu$ is supported on $\{\ldots,-3,-2,-1,1\}$).  
In fact, the weights $q_i$ describe the law of the one-step peeling in the half-plane version of the UIPT, see \cite{Ang05,ACpercopeel}.
 
 We  write $(W_{n})_{n\geq 0}$ for a random walk with values in $ \mathbb{Z}$, started from $W_0=2$ and with jump distribution
$\nu$. Notice that the jumps of $W$ are bounded above by $1$. Furthermore, using \eqref{eq:asympZp} we have for every $n\geq 0$,
 \begin{eqnarray} \label{eq:tailX} \nu(-k)= 2q_{k} \underset{k\to\infty}{\sim}   2\, \mathsf{t}_{\two} k^{-5/2}.  \end{eqnarray}

It follows that $\nu$ is in the domain of attraction of a spectrally negative stable law of index $3/2$. This implies the convergence in distribution in the Skorokhod  sense, 
 \begin{eqnarray} \label{eq:convstable} \left( \frac{W_{[nt]}}{  \mathsf{p}_{\two} \cdot n^{2/3}} \right)_{t \geq 0}  \xrightarrow[n\to\infty]{(d)} (S_{t})_{t \geq 0}, \end{eqnarray} 
 where 
 \begin{eqnarray} \label{def:pfromt} \mathsf{p}_{\two}= \left( \frac{8 \mathsf{t}_{\two} \sqrt{\pi}}{3} \right)^{2/3}= (2/3)^{2/3},  \end{eqnarray}
and  $S$ is the stable L\'evy process with index $3/2$ and no positive jumps, 
 whose distribution is determined by the Laplace transform $\E[\exp(\lambda S_t)]= \exp(t\lambda^{3/2})$ for every $t,\lambda\geq 0$.
Note that the Lévy measure of $S$ is $\frac{3}{4 \sqrt{\pi}} |x|^{-5/2}\,\mathbf{1}_{\{x < 0\}}\, \mathrm{d}x $.

Our first objective is to get a scaling limit analogous to \eqref{eq:convstable} for $(P_n)_{n\geq 0}$. To this end, recall
from \cite[Section VII.3]{Ber96} that one can define a process $(S^+_t)_{t\geq 0}$ with c\`adl\`ag sample paths, which is
distributed as $(S_t)_{t\geq 0}$ ``conditioned to stay positive forever''.
The scaling limit in the following result was suggested in \cite{Ang03} before Lemma 3.1.
To simplify notation we write $\llbracket k,\infty \llbracket=\{k,k+1,k+2,\ldots\}$ and $\rrbracket -\infty,k\rrbracket=\{\ldots,k-2,k-1,k\}$
for every integer $k\in\mathbb{Z}$.

\begin{proposition} \label{prop:positive} {\rm (i)} The Markov chain $(P_{n})_{n \geq 0}$ is distributed as the random walk $(W_{n})_{n \geq 0}$ 
conditioned not to hit $\rrbracket -\infty,1\rrbracket$. Equivalently, $(P_n)_{n\geq 0}$ is distributed as the $h$-transform
of the  random walk $(W_{n})_{n \geq 0}$ killed upon hitting $\rrbracket -\infty,1\rrbracket$, where the function $h$
defined on $\Z$ by
  \begin{equation} \label{def:h}
  h(p) := 
  \left\{\begin{array}{ll}9^{-p}C(p) \quad&\hbox{if }p\geq 2,\\
  0&\hbox{if }p\leq 1,
  \end{array}
  \right.
  \end{equation}
is, up to multiplication by a positive constant, the unique nontrivial nonnegative function that is $\nu$-harmonic on $\llbracket 2,\infty \llbracket$
  and vanishes on $\rrbracket -\infty,1\rrbracket$.
  
  \noindent{\rm(ii)} The following convergence in distribution holds in the Skorokhod sense,
\begin{equation}
\label{conv-peri}
\left( \frac{P_{[nt]}}{  \mathsf{p}_{\two} \cdot n^{2/3}} \right)_{t \geq 0}  \xrightarrow[n\to\infty]{(d)}  \left(S^+_{t}\right)_{t \geq 0}.\end{equation}
where  we recall that $ \mathsf{p}_{\two} = (2/3)^{2/3}$.
\end{proposition}

 \proof (i) Let $h$ be defined by \eqref{def:h}. From the explicit formulas \eqref{probatypeC}
 and \eqref{probatypeR}, one immediately gets that, for every $p\geq 2$ and every $k\in\{-1,1,2,\ldots,p-2\}$,
 \begin{equation}
 \label{eq:h-transform}
 q^{(p)}_k= \frac{h(p-k)}{h(p)}\;q_k.
 \end{equation}
 It then follows from \eqref{def:Pn} and the definition of $\nu$ that, for every $p\geq 2$ and $k\in \{-p+2,-p+3,\ldots,-1,1\}$,
  \begin{eqnarray} 
  \label{harmonicity}
  \mathbb{P}(P_{n+1}=p+k \mid P_{n}=p) = \frac{h(p+k)}{h(p)}\,\nu(k) = \frac{h(p+k)}{h(p)} \mathbb{P}(W_{n+1}=p+k \mid W_{n}=p).
  \end{eqnarray}
  By summing over $k$, we get, for every $p\geq 2$, 
  $$\sum_{k\in \Z} \frac{h(p+k)}{h(p)}\,\nu(k)= 1$$
  so that $h$ is $\nu$-harmonic on $\llbracket 2,\infty \llbracket$. Note that the uniqueness 
  (up to a multiplicative constant) of a positive function 
  that is $\nu$-harmonic on $\llbracket 2,\infty \llbracket$
  and vanishes on $\rrbracket -\infty,1\rrbracket$ is easy, since, for every $p\geq 2$, the value of this function at $p+1$ is determined from
its values for $2\leq i\leq p$. Furthermore, formula \eqref{harmonicity} precisely says that 
  $(P_n)_{n\geq 0}$ is distributed as the $h$-transform
of the  random walk $(W_{n})_{n \geq 0}$ killed upon hitting $\rrbracket -\infty,1\rrbracket$. The fact that this $h$-transform can be interpreted as the random walk $W$
conditioned to stay in $\llbracket 2,\infty \llbracket$ is classical, see e.g. \cite{BD94}.

\noindent{(ii)} This follows from the invariance principle proved in \cite{CC08}.
\endproof

From \eqref{equivalentcp}, we have 
\begin{equation}
h(p) \underset{p\to \infty}{\sim}   \frac{1}{54 \pi \sqrt{3}} \sqrt{p}.
\label{equivalent-h} 
\end{equation}
Still from \eqref{equivalentcp}, we can write, for $p\geq 2$, 
$$h(p)= \frac{1}{3^{7/2}4\sqrt{\pi}}\; \frac{(2p-3)\times(2p-5)\times\cdots\times 3\times 1}
{(2p-4)\times(2p-6)\times\cdots\times 4\times 2},$$
so that $h(p+1)/h(p)= (2p-1)/(2p-2)$, proving that $h$ is monotone increasing
on $\llbracket 2,\infty \llbracket$. 
Then, for every $j\geq 1$, and every $p$ with $p\geq j+2$,
\begin{equation}
\label{compar1}
q^{(p)}_j=\frac{h(p-j)}{h(p)}\,q_j \leq q_j
\end{equation}
and similarly, for every $p\geq 2$,
\begin{equation}
\label{compar2}
q^{(p)}_{-1}= \frac{h(p+1)}{h(p)}\,q_{-1} \geq q_{-1}.
\end{equation}
These bounds will be useful later.

 \subsection{A few applications}
Let us give a few applications of Proposition \ref{prop:positive}. First, it is easy to recover from this proposition the known fact (see \cite[Claim 3.3]{Ang03}) that the Markov chain
$(P_n)_{n\geq 0}$ is transient, 
 \begin{eqnarray} \label{eq:pn->infty} P_{n}  \xrightarrow[n\to\infty]{a.s.} +\infty.  \end{eqnarray}
 To see this, let $p\geq 2$ and write $\P_p$ for a probability measure under which the random walk $W$
 with jump distribution $\nu$ starts from $p$. For every $y\in\Z$, set 
 $T_y=\min\{n\geq 0: W_n=y\}$. Note that $T_y<\infty$ a.s. because the random
 walk $W$ is recurrent. Similarly, suppose that 
 $\wt T_y$ is distributed under $\P_p$ as the hitting time of $y$ for a
 Markov chain with the same transition kernel as $(P_n)_{n\geq 0}$ but started from $p$.
 Then, standard properties of $h$-transforms give for every $p,y\in\llbracket 2,\infty\llbracket$,
 $$\P_p(\wt T_y<\infty)=\frac{h(y)}{h(p)}\, \P_p(W_k\geq 2, \;\forall k\leq T_y).$$
 Since $h$ is monotone increasing on $\llbracket 2,\infty \llbracket$, the right-hand side is smaller than $1$ when $p>y$, 
 giving the desired transience. 
 
 The following corollary was conjectured in \cite[Section 5.1]{BCsubdiffusive}.
 
 \begin{corollary} \label{cor:bouffetout} Any peeling $( \mathsf{T}_{n})_{n \geq 0}$ of the UIPT will eventually discover $T_{\infty}$ entirely, that is 
 $$ \bigcup_{n \geq 0} \mathsf{T}_{n} = T_{\infty}, \quad a.s.$$ \end{corollary}
 
 \proof It is enough to prove that, if $n_0\geq 1$ is fixed, then a.s. every vertex 
 of $\partial \mathsf{T}_{n_0}$ belongs to the interior of $\mathsf{T}_{n_1}$
 for some $n_1>n_0$ sufficiently large. Indeed, if this property holds,
 an inductive argument shows that the minimal distance between a vertex outside $\mathsf{T}_n$ and the root tends to infinity as $n\to\infty$, which gives the
 desired result.
 
 So let us fix $n_0$ and a vertex $v$ of $\partial \mathsf{T}_{n_0}$, and argue 
 conditionally on $\mathsf{T}_{n_0}$ and $v$. We note that, for every 
 $n\geq n_0$, conditionally on
 the event that $v$ is still on the boundary of $\mathsf{T}_{n}$, the probability
 that $v$ will be  ``surrounded'' by the revealed triangle at step $n+1$, and therefore
 will belong to the interior of $\mathsf{T}_{n+1}$, is at least
 $$ \sum_{k= [P_{n}/2]+1}^{P_{n}-2} q_{k}^{(P_{n})} $$
 with the convention that the sum is $0$ if $[P_{n}/2]+1>P_n-2$. If $P_n$ is
 large enough, the latter quantity is bounded below by
 $$\sum_{k= [P_{n}/2]+1}^{[3P_{n}/4]} q_{k}^{(P_{n})} =
 \sum_{k= [P_{n}/2]+1}^{[3P_{n}/4]}  \frac{h(P_{n}-k)}{h(P_{n})} q_{k}
 \geq c\, P_n^{-3/2},$$
 where $c$ is a positive constant and we used \eqref{eq:asympZp}  and 
 \eqref{equivalent-h} in the last inequality. Recalling that
 $P_n\to\infty$ a.s., we see that the proof will be complete
 if we can verify that the series
 $$\sum_{n=1}^\infty P_n^{-3/2}$$
 diverges a.s.
 
 To this end, we argue by contradiction and assume that we can find two constants
 $M<\infty$ and $\ve>0$ such that the probability of the event
 $$\Big\{\sum_{n=1}^\infty P_n^{-3/2}\leq M\Big\}$$
 is greater than $\ve$. On this event, for any $t >1$ and any $n\geq 1$, we have 
$$\int_1^t { \mathrm{d}u} \,\Big(\frac{P_{[nu]}}{n^{2/3}}\Big)^{-3/2}
\leq \frac{1}{n} \sum_{i=n}^{[nt]} \left( \frac{P_{i}}{n^{2/3}}\right)^{-3/2} = \sum_{i=n}^{[nt]} {P_{i}^{-3/2}} \leq M.$$
Using the convergence of Proposition \ref{prop:positive} (ii), we obtain that, for every $t >1$,  the probability of the event $\{\int_{1}^t { \mathrm{d}u}\,(S_{u}^+)^{-3/2} \leq (\mathsf{p}_{\two})^{-3/2} \,M \}$ is greater than $\ve$. Letting $t \to \infty$ we get that 
$$ \mathbb{P}\left( \int_{1}^\infty \frac{ \mathrm{d}u}{(S_{u}^+)^{3/2}} \leq (\mathsf{p}_{\two})^{-3/2} \,M \right) \geq \varepsilon.$$
This is a contradiction because
$$\int_{1}^\infty \frac{ \mathrm{d}u}{(S_{u}^+)^{3/2}}=\infty\;\quad\hbox{a.s.}$$
as can be seen by an application of Jeulin's lemma \cite[Proposition 4 c)]{Jeu82}, noting that
we have $(S_{u}^+)^{-3/2} \build{=}_{}^{(d)} u^{-1} (S^+_1)^{-3/2}$ 
by scaling and that the law of $S^+_1$ is diffuse, for instance by \cite[Corollary VII.16]{Ber96}. \endproof

The next lemma will be an important tool in the proof of Theorems \ref{thm:scalinghull} and \ref{thm:scalingfpp}.

\begin{lemma} \label{lem:1/P} There exist two constants $0<c_{1} <c_{2}<\infty$ such that, for all $n \geq 1$, we have
$$ c_{1} n^{-2/3} \leq \mathbb{E}\Big[\frac{1}{P_n}\Big] \leq c_{2} n^{-2/3}.$$
\end{lemma}

\proof

The lower bound is easy since Proposition \ref{prop:positive} (ii) gives
$$ \mathbb{E}\left[ \frac{n^{2/3}}{P_{n}}\right] \geq \E\Big[\frac{n^{2/3}}{P_n}\wedge 1\Big] \build\longrightarrow_{n\to\infty}^{} \E
\Big[\frac{1}{\mathsf{p}_{\two} S^+_1}\wedge 1\Big] >0.$$
To prove the upper bound, we first fix $k \geq 2$
and $n\geq 1$, and we evaluate $ \mathbb{P}(P_{n} =k)$. By Proposition \ref{prop:positive}~(i) and properties of
$h$-transforms, we have
$$\mathbb{P}(P_{n} =k) =  \frac{h(k)}{h(2)} \cdot  \mathbb{P}( \{W_{i}\geq 2, \forall\,  i \leq n\} \cap \{W_{n} = k\}).$$
We set $\wt W_i=W_n-W_{n-i}$ for $0\leq i\leq n$ and note that we can also define
$\wt W_i$ for $i>n$ in such a way that
$(\wt W_i)_{i\geq 0}$ is a random walk with the same jump distribution
as $W$ and $\wt W_0=0$. We have then
$$ \mathbb{P}(\{W_{i}\geq 2, \forall 0 \leq i \leq n\} \cap \{W_{n} = k\}) = \mathbb{P}( \{\wt{W}_{n} = k-2\} \cap  \{\wt W_{i} \leq k-2, \forall\,  i \leq n\} ) = \frac{ \mathbb{P}( \wt T_{k-1} = n+1)}{ q_{-1}},$$ 
where we have set $\wt T_{k-1}= \min\{ i \geq 0 : \wt W_{i} = k-1\}$. Note that
$\wt W$ has positive jumps only of size $1$. We can thus use
Kemperman's formula (see e.g.~\cite[p.122]{Pit06}) to get
$$  \mathbb{P}( \wt T_{k-1}=n+1) =  \frac{k-1}{n+1} \;\mathbb{P}(
\wt W_{n+1}=k-1).$$
From the last three displays, we have
$$ \mathbb{P}(P_{n} =k) = \frac{3}{2}\, \frac{h(k)}{h(2)}\,  \frac{k-1}{n+1}\; \mathbb{P}(\wt W_{n+1}=k-1).$$

Using the local limit theorem for random walk in the domain
of attraction of a stable distribution (see e.g. \cite[Theorem 4.2.1]{IL71}), we can find
a constant $c''$ such that 
\begin{eqnarray} \label{eq:locallimit}  \mathbb{P}(\wt W_n=k) \leq c''\,n^{-2/3}, \end{eqnarray}
for every $n\geq 1$ and $k\in\Z$.
Then, for every $n\geq 1$,
 \begin{align*}  \mathbb{E}\Big[\frac{1}{P_n}\Big]
 & =  \mathbb{E} \left[ \frac{1}{ P_{n}} \mathbf{1}_{ \{P_{n} > n^{2/3}\}}\right]+   \mathbb{E}\left[ \frac{1}{ P_{n}} \mathbf{1}_{ \{P_{n} \leq n^{2/3}\}}\right]  \\ 
 & \leq  n^{-2/3}+\sum_{k=1}^{ [ n^{2/3}]} \frac{3}{2}\, \frac{h(k)}{h(2)}\,  \frac{k-1}{n+1}\, \frac{1}{k}\; \mathbb{P}(\wt W_{n+1}=k-1)\\
 &\leq  n^{-2/3} + \frac{3 c''}{2h(2)}\;n^{-5/3}  \sum_{k=1}^{[ n^{2/3}]} h(k).\end{align*} 
 The upper bound of the lemma follows using \eqref{equivalent-h}.
 \endproof

\subsection{The scaling limit of volumes}

\label{sec:volume}

Our goal is now to study the scaling limit of the process $(V_{n})_{n\geq 0}$.  We start with a result similar to \cite[Proposition 6.4]{Ang03} about the 
distribution of the size of a Boltzmann triangulation with a large perimeter. For every $p \geq 2$, we let  $T^{(p)}$ denote a random triangulation of the $p$-gon with Boltzmann distribution. 

\begin{proposition} \label{prop:boltz} Set $ \mathsf{b}_{\two}= \frac{2}{3}$. \begin{enumerate}
\item We have 
$  \mathbb{E}[|T^{(p)}|] \sim \mathsf{b}_{\two} \cdot p^2$ as $p \to \infty$.
\item The following convergence in distribution holds:  
$$ p^{-2}|T^{(p)}| \xrightarrow[p\to\infty]{(d)}  \mathsf{b}_{\two}\cdot \xi,$$
where $ \xi$ is a random variable with density $ \displaystyle \frac{ e^{-1/2x}}{x^{5/2} \sqrt{2\pi}}$ on $ \mathbb{R}_{+}$. 
\end{enumerate}
\end{proposition}

\noindent \textbf{Remark.} We have $ \mathbb{E}[ \xi]=1$ and the size-biased version of the distribution of $ \xi$ (with density $  \frac{ e^{-1/2x}}{x^{3/2} \sqrt{2\pi}} $ on $\mathbb{R}_+$) is the  $1/2$-stable distribution
with Laplace transform $e^{-\sqrt{2\lambda}}$. Consequently, for $\lambda >0$, we have 
$$  \mathbb{E}[ e^{-\lambda \xi}] = (1+ \sqrt{ 2\lambda}) e^{- \sqrt{2 \lambda}}.$$

\proof The first assertion follows from the formula $  \mathbb{E}[|T^{(p)}|]= \frac{1}{3}(p-1)(2p-3)$ for $p \geq 2$ which is
easily derived from the exact formula for the generating
function of the sequence $(\#\mathcal{T}_{n,p})_{n\geq 0}$
found in \cite[Proposition 2.4]{AS03}. See also
\cite[Proposition 3.4]{Ray13}.

For the second assertion, we proceed as in \cite[Proposition 6.4]{Ang03}. From the explicit expressions \eqref{eq:tnpexact} and \eqref{eq:defzp}, an asymptotic expansion using Stirling's formula shows that, for every fixed $x >0$, we have 
$$p^2\,\P(|T^{(p)}|=[p^2x])= p^2 \frac{(2/27)^{[p^2x]}\,\# \mathcal{T}_{[p^2x],p}}{Z(p)}  \xrightarrow[p\to\infty]{} \frac{2 e^{-1/(3x)}}{3  x^{5/2} \sqrt{3\pi}},$$
and the convergence holds uniformly 
when $x$ varies over a compact subset of $ \mathbb{R}_{+}$.  Since the right-hand side of the last display is the density of the variable $2 \xi/3$, the desired result follows. \endproof 

We are now ready to prove Theorem \ref{thm:scalingpeeling}. 

\begin{proof}[Proof of Theorem \ref{thm:scalingpeeling}]
%
We will verify that
\begin{equation}
\label{Th1-1}
\left( \frac{P_{[nt]} }{ \mathsf{p}_{\two} \cdot n^{2/3}}, \frac{V_{[nt]}}{ \mathsf{v}_{\two} \cdot n^{4/3}} \right)_{0\leq t\leq 1}  \xrightarrow[n\to\infty]{(d)}  \left(S^+_{t},Z_{t}\right)_{0 \leq t \leq 1}.
\end{equation}
The statement of  Theorem \ref{thm:scalingpeeling} follows, noting that
there is no loss of generality in restricting the time interval to $[0,1]$. The constant $ \mathsf{v}_{\two}$ will appear below as 
 \begin{eqnarray} \label{def:vfrompb} \mathsf{v}_{\two} = (\mathsf{p}_{\two})^2 \, \mathsf{b}_{\two}. \end{eqnarray}
The convergence of the first component in \eqref{Th1-1} is given
by  Proposition \ref{prop:positive}. We will thus
study the conditional distribution of the second component given 
the first one, and Proposition \ref{prop:boltz} will be our main tool. 
We first note that, for every $n\geq 1$, we can write
$$V_n=|\mathsf{T}_n|= V^*_n + \wt V_n,$$
where $V^*_n$ denotes the number of inner vertices of $\mathsf{T}_n$
that belong to $\partial\mathsf{T}_i$ for some $i\leq n-1$, and $\wt V_n$ is thus 
the total number of inner vertices in the Boltzmann triangulations
that were used to fill in the holes in the case of occurence of events $ \mathsf{L}_k$
or $ \mathsf{R}_k$ at some step $i\leq n$ of the peeling process.
Since $\#(\partial\mathsf{T}_i\backslash\partial\mathsf{T}_{i-1})\leq 1$ for $1\leq i\leq n$,
it is clear that $V^*_n\leq n+2$ for every $n\geq 0$. It follows that
\eqref{Th1-1} is equivalent to the same statement where 
$V_{[nt]}$ is replaced by $\wt V_{[nt]}$. 

Next we can write, for every $k\in\{1,\ldots,n\}$, 
\begin{equation}
\label{holes-filled}
\wt V_k=\sum_{i=1}^k \mathbf{1}_{\{P_i<P_{i-1}\}}\,U_i
,
\end{equation}
where, conditionally
on $(P_0,P_1,\ldots,P_n)$,
 the random variables $U_i$ (for $i$ such that $P_i<P_{i-1}$)
are independent, and $U_i$ is distributed as $|T^{(P_{i-1}-P_i+1)}|$, with
the notation of Proposition \ref{prop:boltz}. 

Fix $\ve >0$ and set, for every $k\in\{1,\ldots,n\}$, 
\begin{equation}
\label{holes-filled2}
\wt V_k^{\leq \ve}=\sum_{i=1}^k \mathbf{1}_{\{0<P_{i-1}-P_i\leq \ve n^{2/3}\}}\,U_i
\;,\quad \wt V_k^{> \ve}=\sum_{i=1}^k \mathbf{1}_{\{P_{i-1}-P_i> \ve n^{2/3}\}}\,U_i\;.
\end{equation}
We first observe that $n^{-4/3} \mathbb{E}[\wt V_n^{\leq \ve}]$ is small uniformly in $n$ when $\ve$ is small. Indeed,
it follows from Proposition \ref{prop:boltz} that there is a constant $C$ such that
$\E[|T^{(p)}|]\leq C\,p^2$ for every $p\geq 2$, which gives
$$ {\E}[ \wt V_{n}^{ \leq \varepsilon} ] 
\leq C\sum_{i=1}^n \E[(P_{i-1}-P_i+1)^2\mathbf{1}_{\{0<P_{i-1}-P_i\leq \ve n^{2/3}\}}].$$
On the other hand, from the bound \eqref{compar1} and \eqref{eq:asympZp}, it
is straightforward to verify that, for every $i\geq 1$ and 
every $p\geq 2$,
$$\E[(P_{i-1}-P_i+1)^2\mathbf{1}_{\{0<P_{i-1}-P_i\leq \ve n^{2/3}\}}\mid P_{i-1}=p]\leq C'\sum_{j=1}^{[\ve n^{2/3}]} (j+1)^2 j^{-5/2}\leq C''\,\sqrt{\ve}\,n^{1/3},$$
with some constants $C'$ and $C''$ independent of $n$ and $\ve$.
By combining the last two displays, we obtain, for every $n\geq 1$,
\begin{equation}
\label{petit-saut}
n^{-4/3}{\E}[ \wt V_{n}^{ \leq \varepsilon} ] \leq CC''\,\sqrt{\ve}.
\end{equation}

Let us turn to $ \wt V_n^{> \ve}$. We
write $s_1,s_2,\ldots$ for the jump times
of $S^+$ before time $1$ listed
in decreasing order of their absolute values. For every 
$n\geq 1$, let $\ell^{(n)}_1,\ldots,\ell^{(n)}_{k_n}$ be all integers
$i\in\{1,\ldots,n\}$ such that $P_{i-1}-P_{i}>0$, listed in decreasing
order of the quantities $P_{i-1}-P_{i}$ (and in the usual order of $\mathbb{N}$
for indices such that $P_{i-1}-P_i$ is equal to a given value). 
For definiteness, we also set $\ell^{(n)}_i=1$ if $i> k_n$. It
follows from \eqref{conv-peri}
that, for every integer $K\geq 1$,
\begin{align}
\label{convjumptimes}
&\big(n^{-1}\ell^{(n)}_1,\ldots,n^{-1}\ell^{(n)}_K, n^{-2/3}(P^{(n)}_{\ell^{(n)}_1}-P^{(n)}_{\ell^{(n)}_1-1}), \ldots,n^{-2/3}(P_{\ell^{(n)}_K}-P_{\ell^{(n)}_K-1})\big)\notag\\
&\qquad \xrightarrow[n\to\infty]{(d)} (s_1,\ldots,s_K,\mathsf{p}_\two\,\Delta S^+_{s_1},\ldots,
 \mathsf{p}_\two\,\Delta S^+_{s_K}),
 \end{align}
and this convergence in distribution holds jointly with \eqref{conv-peri}.
Furthermore, using the conditional distribution of
the variables $U_i$ given  $(P_0,\ldots,P_n)$ and 
Proposition \ref{prop:boltz}, we also get, for every integer $K\geq 1$,
\begin{equation}
\label{convjumps}
\Bigg(\frac{U^{(n)}_{\ell^{(n)}_1}}{(P_{\ell^{(n)}_1}-
P_{\ell^{(n)}_1-1})^2},\ldots,\frac{U^{(n)}_{\ell^{(n)}_K}}{(P_{\ell^{(n)}_K}-
P_{\ell^{(n)}_K-1})^2} \Bigg)
\xrightarrow[n\to\infty]{(d)} \Big( 
\mathsf{b}_\two\,\xi_1,\ldots, \mathsf{b}_\two\,\xi_K\Big),
\end{equation}
where $\xi_1,\xi_2,\ldots $ are independent copies of the variable $\xi$
of Proposition \ref{prop:boltz}. This convergence holds jointly with \eqref{conv-peri} and
\eqref{convjumptimes}, provided that we assume that the sequence $\xi_1,\xi_2,\ldots $
is independent of $S^+$. Now note that we can choose $K$ sufficiently
large so that the probability that $|\Delta S^+_{s_K} |< \ve/(2\mathsf{p}_\two)$
is arbitrarily close to $1$. Recalling the definition of $ \wt V_n^{> \ve}$, we 
can combine \eqref{convjumptimes} and \eqref{convjumps} in order to get
the convergence
\begin{equation}
\label{conv-approxi}
\Big(n^{-2/3} P_{[nt]}, n^{-4/3}  \wt V^{> \ve}_{[nt]}\Big)_{0\leq t\leq 1}
\xrightarrow[n\to\infty]{(d)}
\Big(\mathsf{p}_\two\,S^+_t,(\mathsf{p}_\two)^2\mathsf{b}_\two\, Z^\ve_t\Big)_{0\leq t\leq 1},
\end{equation}
where the process $(Z^\ve_t)_{0\leq t\leq 1}$ is defined by
$$Z^\ve_t=\sum_{i=1}^\infty \mathbf{1}_{\{s_i\leq t,\,|\Delta S^+_{s_i} |> \ve/\mathsf{p}_\two\}}\,
(\Delta S^+_{s_i})^2\,\xi_i.$$
In agreement with the notation of the introduction, set, for every $0\leq t\leq 1$,
$$Z_t=\sum_{i=1}^\infty \mathbf{1}_{\{s_i\leq t\}}\,
(\Delta S^+_{s_i})^2\,\xi_i.$$
Then, it is easy to verify that, for every $\delta>0$, 
$$\P\Big(\sup_{0\leq t\leq 1} |Z_t-Z^\ve_t| >\delta\Big) 
\xrightarrow[\ve\to 0]{} 0.$$
Furthermore, \eqref{petit-saut} also gives
$$\sup_{n\geq 1}  \mathbb{P}\Big(\sup_{0\leq t\leq 1}|n^{-4/3}\wt V_{[nt]} -n^{-4/3} \wt V^{\geq \ve}_{[nt]}|>\delta\Big) \xrightarrow[\ve\to 0]{} 0.$$
The convergence \eqref{Th1-1}, with $V$ replaced by $\wt V$, follows from
\eqref{conv-approxi} and the preceding considerations. This completes the proof.
\end{proof}

\section{Distances in the peeling process}
\subsection{Peeling by layers} 
\label{sec:layers}
In this section, we focus on a particular peeling algorithm, which we call the peeling by layers. As previously, we start from the trivial
triangulation that consists only of the root edge. At the first step,  we discover the triangle on the left side
of the root edge to get $\mathsf{T}_1$. To get $\mathsf{T}_2$, we then discover the triangle on the right side of the root edge. Then 
we continue by induction in the following way. We note that the triangle revealed at step $n$ has either one or two edges in the boundary of
$\mathsf{T}_n$. If it has one edge in the boundary, we discover at step $n+1$ the triangle incident to this edge which is not
already in $\mathsf{T}_n$. If it has two edges in the boundary, we do the same for the right-most among these two edges (this makes
sense because in that case the boundary of $\mathsf{T}_n$ must contain at least $3$ edges). See Fig.~\ref{fig:layers} for an example.

This algorithm is particularly well suited to the study of distances from the root vertex, for the following reason. One easily 
proves by induction that, for every $n\geq 1$, one
and only one of the two following possibilities occurs. Either all vertices of $\partial \mathsf{T}_{n}$ are at the same 
distance $h$ from the root vertex. Or there is an integer $h\geq 0$ such that $\partial \mathsf{T}_{n}$  contains both vertices at distance $h$ and 
at distance $h+1$ from the root vertex. In the latter case, vertices at distance $h$ form a connected subset of $\partial \mathsf{T}_{n}$,
and the edge that will be ``peeled off'' at step $n+1$ is the only edge of the boundary whose left end is at distance $h+1$ and whose right end is at distance $h$. 
In both cases we write $H_n=h$, so that the boundary $\partial \mathsf{T}_{n}$ does contain vertices at distance $H_n$
and may also contain vertices at distance $H_n+1$. 
We also set $H_0=0$ by convention.

 \begin{figure}[!h]
 \begin{center}
 \includegraphics[width=15cm]{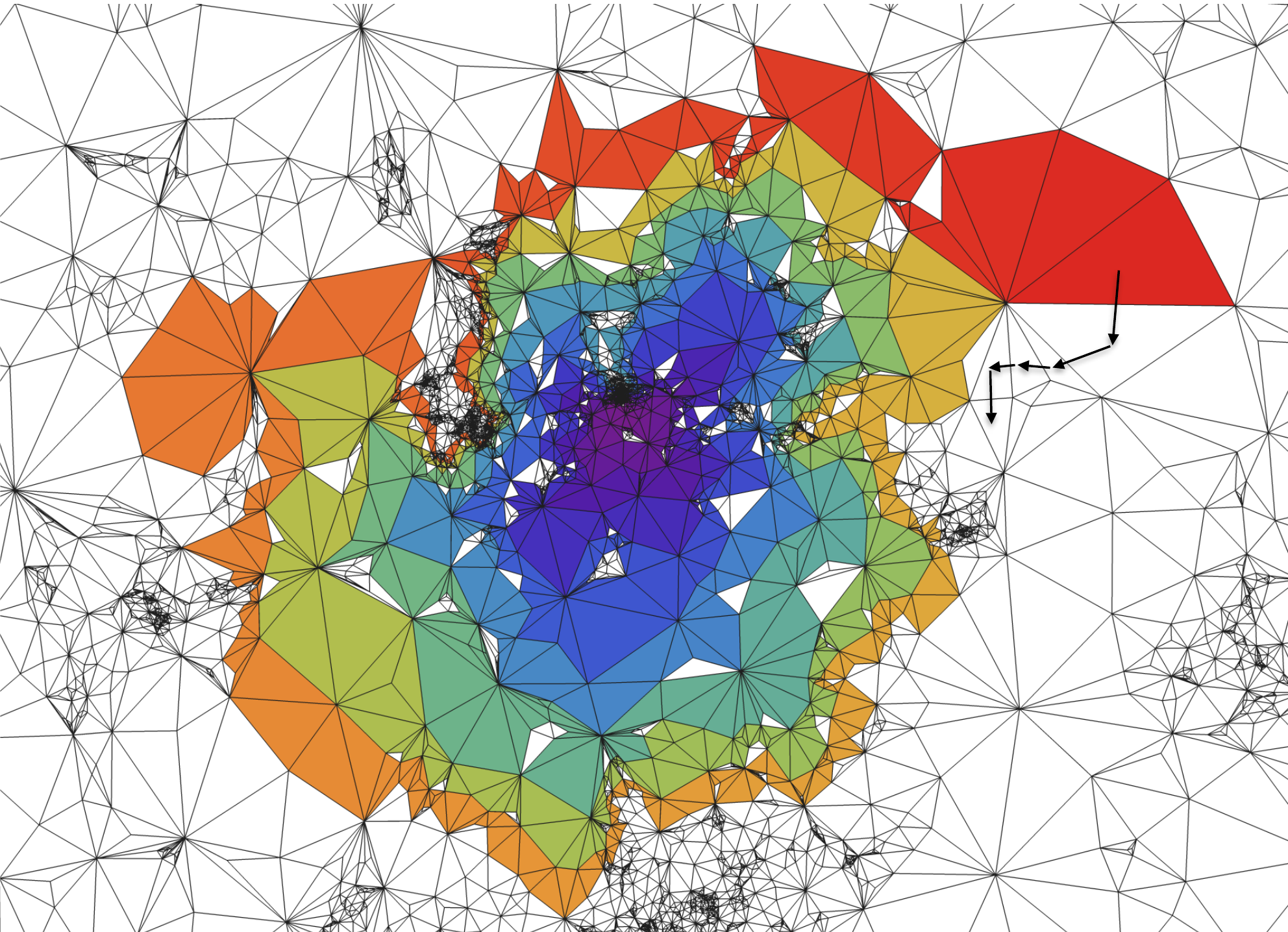}
 \caption{ \label{fig:Budd} The peeling by layers algorithm in a random triangulation drawn in the plane via Tutte's barycentric embedding. 
 The successive layers are represented with different colors. Courtesy of Timothy Budd. See \textrm{https://www.youtube.com/watch?v=afR9yo1P9vE} for the associated movie. }
 \end{center}
 \end{figure}

Since the peeling algorithm discovers the whole triangulation $\mathsf{T}_\infty$
(Corollary \ref{cor:bouffetout}), it is clear that $H_n$ tends to $\infty$
as $n\to\infty$. Also obviously $0\leq H_{n+1}-H_n\leq 1$ for every $n\geq 1$, hence we may set $\sigma_r:=\min\{n\geq 0 : H_n=r\}$
for every integer $r\geq 1$. A simple argument shows that for $n=\sigma_r$, all vertices of
$\partial \mathsf{T}_{n}$ are at distance $r$ from the root vertex (this however does not characterize
$\sigma_r$ since there may exist other times $n>\sigma_r$ with the same property). Furthermore, any
vertex lying outside $\mathsf{T}_{\sigma_r}$ must be at distance at least $r+1$ from the root vertex, and any
triangle of $\mathsf{T}_{\sigma_r}$ that is incident to an edge of the boundary contains a vertex 
at distance $r-1$ from the root vertex (indeed this triangle has been discovered by the peeling algorithm at a 
time where the boundary still contained vertices at distance $r-1$, and the corresponding peeled edge
had to connect a vertex at distance $r$ to a vertex at distance $r-1$). It follows from the previous
considerations that we have $\mathsf{T}_{\sigma_r}=B^\bullet_r(T_{\infty})$ for every $r\geq 1$. Furthermore, 
for every $n\geq 1$ such that $H_n>0$, we have $\sigma_{H_n}\leq n < \sigma_{H_n+1}$ and therefore
 \begin{eqnarray} \label{eq:nesting} B^\bullet_{H_{n}}(T_{\infty}) \subset \mathsf{T}_{n} \subset B^\bullet_{H_{n}+1}(T_{\infty}).  \end{eqnarray}
 This also holds for $n$ such that $H_n=0$, provided we define $B^\bullet_{0}(T_{\infty})$ as the trivial triangulation
 consisting only of the root edge. 

An important consequence is the following fact, which needs not be true for a 
general peeling algorithm. If $\f_n$ stands for the $\sigma$-field generated by 
$\mathsf{T}_0,\mathsf{T_1},\ldots,\mathsf{T}_n$, then the graph distances of vertices of $\mathsf{T}_n$ from the
root vertex are measurable with respect to $\f_n$. This is clear since \eqref{eq:nesting} shows that a geodesic from
any vertex of $\mathsf{T}_n$ to the root visits only vertices of $\mathsf{T}_n$.
 
At an intuitive level, the peeling algorithm ``turns'' around the boundary of the hull of balls of the UIPT in clockwise order and discovers $T_{\infty}$ layer after layer. When turning around $ \partial B_{r}^\bullet(T_{\infty})$, the peeling process creates new vertices at distance $r+1$ from the root vertex in a way similar to a front propagation. See  Fig.~\ref{fig:layers}.

\begin{figure}[!h]
 \begin{center}
 \includegraphics[width=16cm]{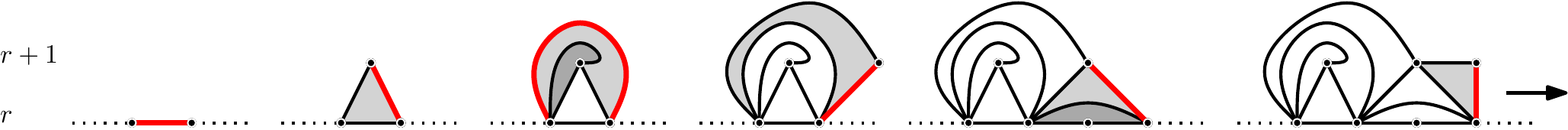}
 \caption{Illustration of the peeling by layers. When $B^\bullet_{r}(T_{\infty})$ has been discovered, we turn around the boundary $\partial B_{r}^\bullet(T_{\infty})$ from left to right in order to reveal the next layer and obtain $B_{r+1}^\bullet(T_{\infty})$. \label{fig:layers}}
 \end{center}
 \end{figure}
 
To simplify notation, we write $B_{r}^\bullet$ and $ \partial B_{r}^\bullet$ instead of $ B_{r}^\bullet(T_{\infty})$ and $ \partial B_{r}^\bullet(T_{\infty})$ in this section. As \eqref{eq:nesting}
suggests, 
the proof of Theorem \ref{thm:scalinghull} will rely on the convergence in distribution 
of a rescaled version of the process $H_n$. Let us sketch some ideas of the proof of the latter convergence. Between times 
$\sigma_r$ and $\sigma_{r+1}$, the peeling process needs to turn around $ \partial B^\bullet_{r}$, which roughly takes a time linear in $| \partial B_{r}^\bullet|$ (see Proposition \ref{prop:A->1/3} below for a precise statement). We 
thus expect that, for some positive constant $a$, 
\begin{eqnarray}  \sigma_{r+1}-\sigma_r \approx \frac{1}{a}\, | \partial B_{r}^\bullet| = \frac{1}{a}P_{\sigma_r} \label{eq:deltaH} \end{eqnarray} 
and therefore 
$$\sigma_r \approx \frac{1}{a}\sum_{i=1}^{r-1} P_{\sigma_i}.$$
A formal inversion now gives for $k$ large,
$$H_k=\sup\{r\geq 0: \sigma_r\leq k\} \approx a\,\sum_{i=1}^k \frac{1}{P_i},$$
and the limit behavior of the right-hand side can be derived from the fact that $(n^{-2/3} P_{[nt]})_{t\geq 0}$ 
converges in distribution to $(\mathsf{p}_{\two}\,S^+_t)_{t\geq 0}$ (Proposition \ref{prop:positive}).  The following proposition shows that the previous heuristic considerations are indeed correct with the value of $a$ given by $ \mathsf{a}_{\two} = 1/3$ (note that
$\mathsf{h}_\two$ in Proposition \ref{prop:scalinglayers} below is then equal to $\mathsf{a}_\two/\mathsf{p}_\two$, and see also Proposition \ref{prop:A->1/3}). 

\begin{proposition}[Distances in the peeling by layers]  \label{prop:scalinglayers} We have the following convergence in distribution for the Skorokhod topology 
 \begin{eqnarray*} \left(  \frac{P_{[nt]}}{ \mathsf{p}_{\two} \cdot n^{2/3}}, \frac{V_{[nt]}}{  \mathsf{v}_{\two} \cdot  n^{4/3}}, \frac{H_{[nt]}}{ \mathsf{h}_{\two}\cdot n^{1/3}}\right)_{t \geq 0}  &\xrightarrow[n\to\infty]{(d)}&  \left(S_{t}^+, Z_{t},\int_{0}^t \frac{ \mathrm{d}u}{S_{u}^+} \right)_{t \geq 0},  \end{eqnarray*}
 where $ \mathsf{h}_{\two} = 12^{-1/3}$.
\end{proposition}


Noting that $|B^\bullet_{r}|=V_{\sigma_r}$ and $|\partial B^\bullet_{r}|=P_{\sigma_r}$,
we will derive Theorem \ref{thm:scalinghull} from the last proposition 
via a time change argument in Section \ref{sec:lamperti}. This
derivation involves time-changing the limiting processes $S_{t}^+$ and $ Z_{t}$ by the inverse of the 
increasing process $\int_{0}^t \frac{ \mathrm{d}u}{S_{u}^+} $, which is clearly related to the
Lamperti transformation connecting continuous-state branching processes to
spectrally positive L\'evy processes. In the next section, we state and prove Proposition 
\ref{prop:A->1/3}, which is the key ingredient of the proof of Proposition \ref{prop:scalinglayers}.
The latter proof will be given in Section \ref{sec:dist-layers}.

\subsection{Turning around layers}

\label{sec:turn-layers}

We write $ \mathscr{L}$ for the set of all {edges} of $T_{\infty}$ that are part of $ \partial B_{r}^\bullet$ for some integer $r \geq 1$. 
Note that all these {edges} belong to $\partial\mathsf{T}_n$ for some $n\geq 1$ (because we know that $B^\bullet_r=\mathsf{T}_{\sigma_r}$ for every $r\geq 1$), but the converse is not true.
For every $n\geq 0$, we write $A_n$ for the number of {edges of $ \mathscr{L}$ belonging to $ \mathsf{T}_{n} \backslash \partial \mathsf{T}_{n}$.}

Clearly, $(A_n)_{n\geq 0}$ is an increasing process. Also, recalling our notation $\f_n$ for the $\sigma$-field generated by 
$\mathsf{T}_0,\mathsf{T_1},\ldots,\mathsf{T}_n$, the random variable $A_n$
is measurable with respect to $\f_n$. The point is that, on one hand, the hulls $B^\bullet_1,\ldots,B^\bullet_{H_n}$ are measurable 
functions of $\mathsf{T}_n$, and, on the other hand, {edges} of $\mathsf{T}_{n} \backslash \partial \mathsf{T}_{n}$ which may be in $ \mathscr{L}$ (i.e.~which link two vertices at the same distance from the root) 
are at distance at most $H_n$ from the root (here it is important that we 
considered only  {edges  of $\mathsf{T}_n\backslash \partial \mathsf{T}_{n}$
in the definition of $A_n$, since the $\sigma$-field $\f_n$ does not give enough information to decide whether
an edge of $\partial \mathsf{T}_{n}$ linking two vertices at distance $H_n+1$ from the root belongs to $ \mathscr{L}$ or not}).

\begin{proposition} \label{prop:A->1/3} We have 
$$ \frac{A_{n}}{n} \xrightarrow[n\to\infty]{(P)}  \frac{1}{3}=:\mathsf{a}_{\two}.$$
\end{proposition}

\proof We use the notation $\Delta A_n= A_{n+1}-A_n$ for every $n\geq 0$. {We note that the inner edges of the Boltzmann triangulations that are used to fill in the holes created by the peeling algorithm cannot be in $ \mathscr{L}$,
and it follows that we have 
 \begin{eqnarray} \label{eq:bounddeltaA_{n}} 0\leq \Delta A_n\leq (\Delta P_{n})_- +1  \end{eqnarray} for every $n\geq 0$, the 
 additional term $1$ coming from the fact that the edge that is peeled at time $n$ could actually be in $ \mathscr{L}$ (this happens only at  times of the form $n = \sigma_{r}$)}. In particular $ \mathbb{E}[\Delta A_n]<\infty$
and $\mathbb{E}[A_n]<\infty$. We then set, for every $i\geq 0$,
$$ \eta_{i} = \mathbb{E}[\Delta A_{i} \mid \mathcal{F}_{i}],$$ 
so that $M_{n} := A_{n}- \sum_{i=0}^{n-1} \eta_{i}$ is a martingale with respect to the filtration $(\f_n)$. 

We first prove that $M_{n}/n \to 0$ in probability. To this end,  we use bounds on the second moment of $\Delta M_{n}$. 
Recall our bound {$\Delta A_n\leq (\Delta P_{n})_-+1$}, and note that, for every $k\geq 1$ and every $p\geq 2$, 
\eqref{eq:tailX} and \eqref{compar1} give
$$ \mathbb{P}( \Delta P_{n} =-k \mid P_{n}=p) = \frac{h(p-k)}{h(p)} \mathbb{P}( \Delta W_{n} =-k) \leq C k^{-5/2},$$ 
for some constant $C>0$ independent of $p$ and $k$. It follows that
$$ \mathbb{E}[(\Delta A_{n})^2 \mid P_{n}=p]\leq   1+ C \sum_{k=1}^{p-2} {(k+1)^2 k^{-5/2}}  = O(\sqrt{p}).  
$$
Since $P_n \leq n+2$, we deduce from the last display that 
$$ \mathbb{E}[(\Delta M_{n})^2] =  \mathbb{E}[(\Delta A_{n} - \eta_{n})^2] \leq 2\left( \mathbb{E}[(\Delta A_{n})^2] +  \mathbb{E}[  \mathbb{E}[\Delta A_{n} \mid \mathcal{ F}_{n}]^2]\right) {\leq } 4 \mathbb{E}[( \Delta A_{n})^2]= O(\sqrt{n}).$$  
Since the martingale $M$ has orthogonal increments, we get  $ \mathbb{E}[M_{n}^2] =O(n^{3/2})$ and it follows 
that $M_{n}/n \to 0$ in $L^2$. 

To complete the proof of Proposition \ref{prop:A->1/3}, it is then enough to verify 
that
\begin{equation}
\label{conv-condi-means}
  \frac{1}{n}\sum_{i=0}^{n-1} \eta_{i}  \xrightarrow[n\to\infty]{(P)} \frac{1}{3}.
  \end{equation}

The idea of the proof is as follows. For most times $n$, the boundary $\partial \mathsf{T}_{n} $ has both a ``large'' number of
vertices at distance $H_n$ and a ``large'' number of
vertices at distance $H_n+1$ from the root. Then, except on a set of small probability, the only events 
leading to a nonzero value of $\Delta A_n$ are events of type $\mathsf{R}_k$ for which  \begin{eqnarray} \label{eq:heuristicdeltaAn} \Delta A_n=-\Delta P_n = k.  \end{eqnarray} The conditional
expectation of $\Delta A_n$ is thus computed using the probabilities of the events $\mathsf{R}_k$. 

To make the preceding argument rigorous, we introduce some notation. 
For every integer $n\geq 0$, write $U_n$
for the number of vertices in $\partial  \mathsf{T}_n$
that are at distance $H_n$ from the root vertex. Note that the function $n\mapsto U_n$
is nonincreasing on every interval $[\sigma_r,\sigma_{r+1}[$
where $H_n$ is equal to $r$. We also set $G_n=P_n-U_n$, which represents 
the number of vertices in $\partial  \mathsf{T}_n$
that are at distance $H_n+1$ from the root vertex.

\begin{lemma}
\label{ideal-envir}
For every integer $L\geq 1$, we have 
$$\frac{1}{n} \sum_{i=0}^{n} \mathbf{1}_{\{U_i\leq L\;{\it or}\;G_i\leq L\}}
 \xrightarrow[n\to\infty]{(P)}  0.$$
\end{lemma}

Let us postpone the proof of this lemma. To complete the proof of \eqref{conv-condi-means}, we first use the bound 
\eqref{compar1} to deduce from
the inequality {$\Delta A_n\leq |\Delta P_n| +1$} that, for every $n\geq 0$,
\begin{equation}
\label{unifo-bound}
\eta_n = \mathbb{E}[\Delta A_{n} \mid\f_n] \leq \mathbb{E}[|\Delta P_n| \mid \f_n] +1 \leq C_1,
\end{equation}
for some finite constant $C_1$. Furthermore, using \eqref{compar1} again, we have also, for every integer $L\geq 1$,
\begin{equation}
\label{large-bound}
\mathbb{E}[\Delta A_{n}\,\mathbf{1}_{\{|\Delta P_n|\geq L\}} \mid\f_n] \leq \mathbb{E}\big[{\big(|\Delta P_n|+1 \big)}\,\mathbf{1}_{\{|\Delta P_n|\geq L\}} \mid \f_n\big] \leq
c_{(L)}
\end{equation}
where the constants $c_{(L)}$ are such that $c_{(L)}\to 0$ as $L\to\infty$. Then, on the event $\{U_n\geq L,G_n\geq L\}$, the condition $|\Delta P_n|< L$ ensures that the only transitions of the peeling algorithm at step $n+1$ leading to
a positive value of $\Delta A_n$ are of type $\mathsf{R}_k$ for some $k$, and in that case $\Delta A_n=-\Delta P_n=k$. It follows that, still
on the event $\{U_n\geq L,G_n\geq L\}$,
 \begin{eqnarray} \label{eq:apparitionde1/3}\mathbb{E}[\Delta A_{n}\,\mathbf{1}_{\{|\Delta P_n|< L\}} \mid\f_n]= \sum_{k=1}^{L-1} k\,q^{(P_n)}_k
\leq \sum_{k=1}^\infty k\,q_k = \frac{1}{3}.  \end{eqnarray}
Note that we have $P_n\geq 2L$ on the event $\{U_n\geq L,G_n\geq L\}$. Since $q^{(p)}_k$ converges to $q_k$ 
as $p\to\infty$, the preceding considerations and \eqref{large-bound} entail that, for every $\ve>0$, we can fix $L_0>0$ so that, for every $L\geq L_0$ 
and every $n$, we have, on the event $\{U_n\geq L,G_n\geq L\}$,
\begin{equation}
\label{ideal-bound}
\frac{1}{3}-\ve \leq \mathbb{E}[\Delta A_{n}\mid\f_n] \leq \frac{1}{3}+\ve.
\end{equation}
Finally, we have, using \eqref{unifo-bound},
$$\Big|  \frac{1}{n}\sum_{i=0}^{n-1} \eta_{i} -  \frac{1}{n}\sum_{i=0}^{n-1}\mathbf{1}_{\{U_i\geq L,G_i\geq L\}}\,\mathbb{E}[\Delta A_{i}\mid\f_i]\Big|
\leq \frac{C_1}{n} \sum_{i=0}^{n-1}\mathbf{1}_{\{U_i\leq L\;{\rm or}\;G_i\leq L\}},$$
and we can now combine \eqref{ideal-bound} and Lemma \ref{ideal-envir} to get 
our claim \eqref{conv-condi-means}. This completes the proof of Proposition \ref{prop:A->1/3}, but we still have to prove
Lemma \ref{ideal-envir}. \endproof

\proof[Proof of Lemma \ref{ideal-envir}] We start with some preliminary observations.
From the definition of the peeling by layers, one easily checks that the
triple $(P_n,G_n,H_n)_{n\geq 0}$ is a Markov chain 
with respect to the filtration $(\f_n)$, taking values
in $\{(p,\ell,h)\in\Z^3: p\geq 2, 0\leq \ell\leq p-1, h\geq 0\}$, and whose transition
kernel $Q$ is specified as follows:
\begin{equation}
\label{transi}
\begin{array}{ll}
Q((p,\ell,h),(p+1,\ell+1,h))= q^{(p)}_{-1}\qquad&\\
Q((p,\ell,h),(p-k,\ell-k,h))= q^{(p)}_k\qquad&\hbox{for }1\leq k\leq \ell-1\\
Q((p,\ell,h),(p-k,\ell,h))= q^{(p)}_k\qquad&\hbox{for }1\leq k\leq p-\ell-1\\
Q((p,\ell,h),(p-k,0,h))= q^{(p)}_k\qquad&\hbox{for }\ell\leq k\leq p-2\\
Q((p,\ell,h),(p-k,0,h+1))= q^{(p)}_k\qquad&\hbox{for }p-\ell\leq k\leq p-2\,.
\end{array}
\end{equation}
The Markov chain $(P_n,G_n,H_n)_{n\geq 0}$ starts from the initial value
$(2,1,0)$. 

Obviously, the 
triple $(P_n,U_n,H_n)_{n\geq 0}$ is also a Markov chain, now with values
in $\{(p,\ell,h)\in\Z^3: p\geq 2, 1\leq \ell\leq p, h\geq 0\}$, and its transition
kernel $Q'$ is expressed by the formula analogous to \eqref{transi},
where only the first and the last two lines are different and replaced by
\begin{equation}
\label{transi2}
\begin{array}{ll}
Q'((p,\ell,h),(p+1,\ell,h))= q^{(p)}_{-1}\qquad&\\
Q'((p,\ell,h),(p-k,p-k,h+1))= q^{(p)}_k\qquad&\hbox{for }\ell\leq k\leq p-2\\
Q'((p,\ell,h),(p-k,p-k,h))= q^{(p)}_k\qquad&\hbox{for }p-\ell\leq k\leq p-2\,.
\end{array}
\end{equation}

We now fix $k\in\{0,1,\ldots,L\}$. We will prove that
\begin{equation}
\label{claim1}
\frac{1}{n} \sum_{i=0}^{n} \P(G_i=k)
\build{\la}_{n\to\infty}^{} 0.
\end{equation}
Let us explain why the lemma follows from \eqref{claim1}. If 
$k'\in\{1,\ldots,L\}$, a simple argument using the Markov chain
$(P_n,U_n,H_n)$ shows that, for every $i\geq 1$,
$$\P( G_{i+1}= 0 \mid \f_i) \geq q^{(P_i)}_{k'}\,\mathbf{1}_{\{U_i=k'\}}\,\mathbf{1}_{\{P_i\geq k'+2\}}$$
and therefore
$$\P( G_{i+1}= 0)\geq 
\beta \,\P(U_i=k',P_i\geq k'+2),$$
with a constant $\beta>0$ depending
on $k'$. If we assume that \eqref{claim1} holds for $k=0$, the latter bound (together with the transience
of the Markov chain $(P_n)$) implies that
\begin{equation}
\label{claim2}
\frac{1}{n} \sum_{i=0}^{n} \P(U_i=k')
\build{\la}_{n\to\infty}^{} 0.
\end{equation}
Clearly the lemma follows from \eqref{claim1} and \eqref{claim2}.

Let us prove \eqref{claim1}. Let $N\geq 1$, and write $T^N_1,T^N_2,\ldots$ for the successive
passage times of the Markov chain $(P_n,G_n,H_n)$
in the set $\{(p,\ell,h): p\geq N, \ell=k\}$. We claim
that  there exist two
positive constants $c$ and $\alpha$ 
(which depend on $k$ but not on $N$) such that, 
for every sufficiently large $N$ and for every 
integer $i\geq 1$,
\begin{equation}
\label{claim3}
\P[T^N_{i+1}-T^N_i \geq \alpha N\mid \f_{T^N_i}] \geq c.
\end{equation}
If the claim holds, simple arguments show that we have a.s.
$$\liminf _{j\to\infty} \frac{T^N_j}{j} \geq \alpha c N$$
and it follows that, a.s.,
$$\limsup_{n\to\infty} \frac{1}{n} \sum_{i=0}^n \mathbf{1}_{\{P_i\geq N, G_i=k\}}
\leq \frac{1}{\alpha c N}.$$
We can remove $P_i\geq N$ in the indicator function
since the Markov chain $(P_n)_{n\geq 0}$ is transient.
This gives \eqref{claim1} since $N$ can be taken arbitrarily large.

Let us verify the claim. Applying the strong Markov property at
time $T^N_i$ leads to a Markov chain $(\wt P_n,\wt G_n,\wt H_n)$
with transition kernel $Q$ but now started from some triple $(p_0,\ell_0,h_0)$
such that $p_0\geq N$ and $\ell_0=k$. We also set $\wt U_n=\wt P_n-\wt G_n$. 
The bound \eqref{claim3} reduces to finding two positive constants
$\alpha$ and $c$ such that, for every sufficiently large $N$,
\begin{equation}
\label{claim3reduced}
\P(\tau_k \geq \alpha N)\geq c,
\end{equation}
where $\tau_k=\min\{j\geq 1: \wt G_j=k\}$.
We set $\wt T:=\inf\{n\geq 0: \wt P_n=\wt U_n\}$, and observe that we have either
$\wt H_{\wt T}=h_0+1$ or $\wt H_{\wt T}= h_0$. 

By looking at the transition kernel $Q$ and using the bounds \eqref{compar1}
and \eqref{compar2}, we see that we can couple 
the Markov chain $(\wt P_n,\wt G_n,\wt H_n)$ with a random walk $(Y_n)$
started from $\ell_0=k$,
whose jump distribution $\mu$ is given by $\mu(1)= q_{-1}$, $\mu(-j)=q_{j}$
for every $j\geq 1$, and $\mu(0)=1-\mu(1)-\sum_{j\geq 1} \mu(-j)$, in such a 
way that
$$\wt G_n \geq Y_n\;, \hbox{ for every } 0\leq n<\wt T,$$
and on the event where $Y_{1}=k+1$ and $\min_{j\geq 1} Y_j=k+1$ we have $\wt H_{\wt T}=h_0+1$ (the point is that on the latter event, the transition corresponding
to the last line of \eqref{transi2} will not occur, at any time $n$ such that 
$0\leq n<\wt T$). Since the random walk $Y$ has a positive drift to $\infty$, the
latter event occurs with probability $c_0>0$.  We have thus obtained that
\begin{equation}
\label{key0}
\P(\{\wt G_n\geq k+1, \hbox{ for every } 1\leq n<\wt T\}
\cap \{ \wt H_{\wt T}=h_0+1\})\geq c_0.
\end{equation}

Next we observe that there is a positive constant $c_1$ such that,
for every $\ve>0$, we have, for all sufficiently large $N$,
\begin{equation}
\label{key1}
\P(\{\wt T \leq c_1(N-k)\}\cap\{ H_{\wt T}=h_0+1\}) < \varepsilon.
\end{equation}
To get this bound, we now consider the transition kernel $Q'$:
We use \eqref{compar1} to observe that we can couple $(\wt P_n,\wt U_n,\wt H_n)$
with a random walk $Y'$ started from $N-k$, with only nonpositive jumps distributed according
to $\mu'(-k)=q_{k}$ for every $k\geq 1$
(and of course $\mu'(0)=1-\sum_{k\geq 1}\mu'(-k)$), in such a way that
$$\wt U_n \geq Y'_n\;, \hbox{ for every } 0\leq n<\wt T,$$
and $Y'_{\wt T}\leq 0$ on the event $\{\wt H_{\wt T}=h_0+1\}$. 
In particular on the event $\{\wt H_{\wt T}=h_0+1\}$ the hitting time
of the negative half-line by $Y'$ must be smaller than or equal to
$\wt T$. Since $\mu'$ has a finite first moment, the law
of large numbers gives a constant $c_1$ such that \eqref{key1} holds.

By combining \eqref{key0}
and \eqref{key1}, and recalling the definition of $\tau_k$, we get
\begin{align*}
&\P(\tau_k\geq c_1(N-k))\\
&\quad \geq\P(\{\wt G_n\geq k+1, \hbox{ for every } 1\leq n<\wt T\}\cap\{ H_{\wt T}=h_0+1\})
- \P(\{\wt T \leq c_1(N-k)\}\cap\{ H_{\wt T}=h_0+1\})\\
& \quad\geq c_0-\varepsilon,
\end{align*}
Our claim \eqref{claim3reduced} now follows since we can choose $\varepsilon < c_0$. 
\endproof

\subsection{Distances in the peeling by layers}
\label{sec:dist-layers}

We need another lemma before we proceed to the proof of Proposition \ref{prop:scalinglayers}.

\begin{lemma}  \label{lem:H->0}
There exists a constant $C$ such that $\E[H_n]\leq C n^{1/3}$, for every $n\geq 1$. 
\end{lemma}

\proof It will be convenient to introduce a process $H'_n$ which coincides with $H_n$ at times of 
the form $\sigma_r$, $r\geq 1$, but which ``interpolates'' $H_n$ on every interval
$[\sigma_r,\sigma_{r+1}]$. To be specific, we recall the notation introduced in
the proof of Lemma \ref{ideal-envir}, and we set for every $n\geq 0$,
$$H'_{n} = H_{n} + \frac{G_{n}}{P_{n}}.$$
From the form of the transition kernel of the Markov chain $(P_n,G_n,H_n)$ (see 
the proof of Lemma \ref{ideal-envir}), we get, for every triple $(p,\ell,h)$ such that $\P(P_{n}=p, G_{n}= \ell, H_n=h)>0$,
 \begin{eqnarray*} \mathbb{E}\big[|\Delta H'_{n}|\, \big| P_{n}=p, G_{n}= \ell, H_n=h \big] &=
 & q_{-1}^{(p)} \left| \frac{\ell+1}{p+1} - \frac{\ell}{p}\right| 
 + \sum_{k=1}^{p-2} q_{k}^{(p)} \left| \frac{(\ell-k) \vee 0}{p-k}- \frac{\ell}{p}\right|\\ 
 &+& \sum_{k=1}^{p-\ell-1} q_{k}^{(p)} \left| \frac{\ell }{p-k}- \frac{\ell}{p}\right| + \sum_{k=p-\ell}^{p-2}q_{k}^{(p)} \left(1 -\frac{\ell}{p}\right).  \end{eqnarray*}
 Then it is not hard to verify that each term in the right-hand side is bounded above by $c/p$, with some constant
 $c$ independent of $(p,\ell,h)$. Indeed, writing $c$ for a constant that may vary from line to line,
 and using \eqref{compar1}, we have
 $$q_{-1}^{(p)} \left| \frac{\ell+1}{p+1} - \frac{\ell}{p}\right| \leq \frac{1}{p+1},$$
 and similarly,
 \begin{align*}
 &\sum_{k=1}^{\ell} q_{k}^{(p)} \left| \frac{\ell-k}{p-k} - \frac{\ell}{p}\right|=\sum_{k=1}^\ell q_{k}^{(p)}  \frac{k(p-\ell)}{p(p-k)}
 \leq \frac{1}{p} \sum_{k=1}^\infty k\,q_k = \frac{c}{p},\\
&\sum_{k=\ell+1}^{p-2} q_{k}^{(p)}\,\frac{\ell}{p}\leq \frac{1}{p} \sum_{k=\ell+1}^\infty k\,q_k\leq \frac{c}{p},\\
&\sum_{k=1}^{p-\ell-1} q_{k}^{(p)} \left| \frac{\ell }{p-k}- \frac{\ell}{p}\right| 
 = \sum_{k=1}^{p-\ell-1} q_{k}^{(p)} \frac{\ell k}{p(p-k)} \leq \sum_{k=1}^{p-\ell-1} q_{k}^{(p)} \frac{k}{p} \leq \frac{c}{p},\\
&\sum_{k=p-\ell}^{p-2} q_{k}^{(p)}\left(1 -\frac{\ell}{p}\right)\leq \left(1 -\frac{\ell}{p}\right) \sum_{k=p-\ell}^\infty q_k \leq \left(1 -\frac{\ell}{p}\right)
 \times c\,(p-\ell)^{-3/2} = \frac{c}{p}\,(p-\ell)^{-1/2}.
 \end{align*}
We conclude that there exists a constant $C'$ such that $\E[\Delta H'_n\mid \f_n]\leq C'/P_n$. By Lemma \ref{lem:1/P}, we
have then $\E[\Delta H'_n]\leq C''n^{-2/3}$ with some other constant $C''$. It follows that 
$\E[H'_n]\leq C'''n^{1/3}$, giving the bound of the lemma since $H_n\leq H'_n$. 
\endproof

\proof[Proof of Proposition \ref{prop:scalinglayers}] It follows from Theorem \ref{thm:scalingpeeling} and
Proposition \ref{prop:A->1/3}, together with monotonicity arguments for the last component, that we have the joint convergence in distribution 
\begin{equation}
\label{joint-layer}
\Big(n^{-2/3}P_{[nt]}, n^{-4/3}V_{[nt]}, n^{-1} A_{[nt]}\Big)_{t\geq 0}
\xrightarrow[n\to\infty]{(d)} (\mathsf{p}_\two\,S^+_t,\mathsf{v}_\two\,Z_t,\mathsf{a}_\two\,t)_{t\geq 0}
\end{equation}
in the Skorokhod sense.
We now need to deal with the convergence of the 
(rescaled) process $H$. We first note that by construction we have 
 $A_{\sigma_{r+1}}-A_{\sigma_r}= P_{\sigma_r}$ for every $r\geq 1$. More precisely, for every
 $r\geq 1$ and every $n$ with $\sigma_r\leq n<\sigma_{r+1}$, we have
 \begin{align*}
 &A_{\sigma_{r+1}}-A_n= U_n\leq P_n,\\
 &A_n-A_{\sigma_r}= P_{\sigma_r}-U_n\leq P_{\sigma_r}.
 \end{align*}
 It easily follows that, for every $
 0 \leq n_{1} \leq n_{2}$, we have 
  \begin{eqnarray} \label{eq:lowerH}
  \frac{A_{n_{2}}-A_{n_{1}}}{\max_{ n_{1} \leq i \leq n_{2}} P_{i}}\leq H_{n_{2}}-H_{n_{1}}+1,
  \end{eqnarray}
and
 \begin{eqnarray} \label{eq:upperH} H_{n_{2}}-H_{n_{1}} \leq \frac{A_{n_{2}}-A_{n_{1}}}{\min_{ n_{1} \leq i \leq n_{2}} P_{i}} +1.  
 \end{eqnarray}
 
Fix $0<s<t$. By \eqref{joint-layer},
 $$n^{-2/3} \min_{[ns]\leq k\leq [nt]} P_k \xrightarrow[n\to\infty]{(d)} \mathsf{p}_\two \,\inf_{s\leq u\leq t} S^+_u,$$
 and the limit is a (strictly) positive random variable. Using also 
 Proposition \ref{prop:A->1/3}, we then deduce from the bound \eqref{eq:upperH} that the sequence
$n^{-1/3} (H_{[nt]}-H_{[ns]})$
 is tight. Hence we can assume that along a suitable subsequence, for every integer $k\geq 0$,
 for every $1\leq i\leq 2^k$, we have the convergence in distribution
 \begin{equation}
\label{conv-H1}
n^{-1/3}\Big( H_{[n(s+i2^{-k}(t-s))]} - H_{[n(s+(i-1)2^{-k}(t-s))]}\Big) 
\xrightarrow[n\to\infty]{(d)} \Lambda^{(s,t)}_{k,i}
\end{equation}
where $\Lambda^{(s,t)}_{k,i}$ is a nonnegative random variable. Moreover, we can 
assume that the convergences \eqref{conv-H1} hold jointly, and jointly with \eqref{joint-layer}. 
It then follows from the bounds \eqref{eq:lowerH} and \eqref{eq:upperH} that, for every
$k$ and $i$,
$$\frac{ \mathsf{a}_{\two}}{ \mathsf{p}_{\two}}\,\frac{2^{-k}(t-s)}{\displaystyle \sup_{s+(i-1)2^{-k}(t-s)\leq u\leq s+i2^{-k}(t-s)} S^+_u}
\leq \Lambda^{(s,t)}_{k,i}
\leq \frac{ \mathsf{a}_{\two}}{ \mathsf{p}_{\two}}\,\frac{2^{-k}(t-s)}{\displaystyle \inf_{s+(i-1)2^{-k}(t-s)\leq u\leq s+i2^{-k}(t-s)} S^+_u}.
 $$
Note that   $ \mathsf{a}_{\two}/ \mathsf{p}_{\two}=12^{-1/3}=:\mathsf{h}_\two$. By summing over 
$i$, we get
 $$\mathsf{h}_\two\,\sum_{i=1}^{2^k}\frac{2^{-k}(t-s)}{\displaystyle \sup_{s+(i-1)2^{-k}(t-s)\leq u\leq s+i2^{-k}(t-s)} S^+_u}
\leq \Lambda^{(s,t)}_{0,1}
\leq\mathsf{h}_\two\, \sum_{i=1}^{2^k} \frac{2^{-k}(t-s)}{\displaystyle \inf_{s+(i-1)2^{-k}(t-s)\leq u\leq s+i2^{-k}(t-s)} S^+_u}.
 $$
When $k\to\infty$, both the right-hand side and the left-hand-side of the previous display converge a.s. to
$$\mathsf{h}_\two\int_s^t \frac{\mathrm{d}u}{S^+_u}.$$
This argument (and the fact that the limit does not depend on the chosen subsequence) thus gives 
\begin{equation}
\label{conv-H2}
n^{-1/3} (H_{[nt]}-H_{[ns]}) \xrightarrow[n\to\infty]{(d)} \mathsf{h}_\two \int_s^t \frac{\mathrm{d}u}{S^+_u},
\end{equation}
and this convergence holds jointly with \eqref{joint-layer}.

At this point, we use Lemma \ref{lem:H->0}, which tells us that $\E[n^{-1/3} H_{[ns]}]$ can be made arbitrarily small,
uniformly in $n$, by choosing $s$ small. Also Lemma \ref{lem:H->0}, \eqref{conv-H2} and Fatou's lemma  imply that 
$$\E\Big[ \int_s^t \frac{\mathrm{d}u}{S^+_u}\Big]$$
is bounded above independently of $s\in(0,t]$, and therefore $\int_0^t \frac{\mathrm{d}u}{S^+_u}<\infty$ a.s. (we could have obtained 
this more directly). Letting $s \to 0$, we deduce from the previous considerations that 
\begin{equation}
\label{conv-H3}
n^{-1/3} H_{[nt]} \xrightarrow[n\to\infty]{(d)} \mathsf{h}_\two \int_0^t \frac{\mathrm{d}u}{S^+_u},
\end{equation}
jointly with \eqref{joint-layer}. The statement of Proposition \ref{prop:scalinglayers} now follows from monotonicity
arguments using the fact that the limit in \eqref{conv-H3} is continuous in $t$. \endproof

\subsection{From Proposition \ref{prop:scalinglayers} to Theorem \ref{thm:scalinghull}}
\label{sec:lamperti}

In this section, we deduce Theorem \ref{thm:scalinghull} from Proposition \ref{prop:scalinglayers} via a time
change argument. We start with some preliminary observations.

We fix $x>0$ and write $(\Gamma^x_t)_{t\geq 0}$ for the stable L\'evy process with index $3/2$ and no negative jumps
started from $x$, whose distribution is characterized by the formula
$$\E[\exp(-\lambda (\Gamma^x_t-x))]= \exp(\lambda t^{3/2})\;,\qquad \lambda,t\geq 0.$$
Equivalently, $\Gamma^x_t=x-S_t$ where $S_t$ is as in the introduction. Set $\gamma_x:=\inf\{t\geq 0:\Gamma^x_t=0\}$. Then $\gamma_x<\infty$ a.s., and a classical 
time-reversal theorem (see e.g. \cite[Theorem VII.18]{Ber96}) states that the law of
$(\Gamma^x_{(\gamma_x-t)-})_{0\leq t\leq \gamma_x}$  (with $\Gamma^x_{0-}=x$) coincides with the law of 
$(S^+_t)_{0\leq t\leq \rho_x}$, where $\rho_x:=\sup\{t\geq 0:S^+_t=x\}$. 

On the other hand, consider the process $\mathcal{L}$ of Section 1. If $\lambda_x:=\sup\{t\geq 0: \mathcal{L}_t\leq x\}$, then 
$\lambda_x<\infty$ a.s. and setting $X^x_t=\mathcal{L}_{(\lambda_x-t)-}$ 
for $0\leq t\leq \lambda_x$ (with $\mathcal{L}_{0-}=0$), the process $(X^x_t)_{0\leq t\leq \lambda_x}$ is distributed as
the continuous-state branching process with branching mechanism $\psi(u)=u^{3/2}$ started from $x$
and stopped when it hits $0$. See \cite[Section 2.1]{CLGHull} for more details. 

The classical Lamperti transformation asserts that, if we set
$$\tau^x_t:=\inf\{s\geq 0: \int_0^s \frac{\mathrm{d}u}{\Gamma^x_u} \geq t\}$$
for $0\leq t \leq R_x:=\int_0^{\gamma_x} \frac{\mathrm{d}u}{\Gamma_u^x} $, the time-changed process
$(\Gamma^x_{\tau^x_t})_{0\leq t\leq R_x}$ has the same distribution as $(X^x_t)_{0\leq t\leq \lambda_x}$. 
We can then combine the Lamperti transformation with the preceding observations to obtain that, if
$$\eta_t:=\inf\{s\geq 0: \int_0^s \frac{\mathrm{d}u}{S^+_u}\geq t\},$$
for every $t\geq 0$, the process 
$$\Big(S^+_{\eta_t}, 0\leq t\leq \int_0^{\rho_x} \frac{\mathrm{d}u}{S^+_u}\Big)$$
has the same distribution as $(\mathcal{L}_t)_{0\leq t\leq \lambda_x}$. Since this holds for
every $x>0$, we conclude that the processes $(S^+_{\eta_t})_{t\geq 0}$
and $(\mathcal{L}_t)_{t\geq 0}$ have the same distribution. It easily follows that we have also
\begin{equation}
\label{ident-Lamperti}
\Big(S^+_{\eta_t},Z_{\eta_t}\Big)_{t\geq 0}\build{=}_{}^{(d)} (\mathcal{L}_t,\mathcal{M}_t)_{t\geq 0},
\end{equation}
with the notation of Section 1.

Let us turn to the proof of Theorem \ref{thm:scalinghull}. We recall that, for every integer $r\geq 1$,
we have $|\partial B^\bullet_r|=P_{\sigma_r}$ and $|B^\bullet_r|=V_{\sigma_r}$, with
$\sigma_r=\min\{n: H_n\geq r\}$. We use the convergence in distribution of Proposition \ref{prop:scalinglayers}
and the Skorokhod representation theorem to find, for every $n\geq1$, 
a triple $(P^{(n)},V^{(n)},H^{(n)})$ having the same distribution as $(P,V,H)$, in such a way that we now have 
the almost sure convergence
\begin{equation}
\label{a.s.scalinghull}
\Bigg( \frac{P^{(n)}_{[nt]}}{ \mathsf{p}_{\two} \cdot n^{2/3}},\frac{V^{(n)}_{[nt]}}{ \mathsf{v}_{\two} \cdot n^{4/3}},
\frac{H^{(n)}_{[nt]}}{ \mathsf{h}_{\two}\cdot n^{1/3}}\Bigg)_{t \geq 0}   \xrightarrow[n\to\infty]{a.s.} \left(S^+_{t}, Z_{t},\int_{0}^t \frac{ \mathrm{d}u}{S_{u}^+}\right)_{t \geq 0},
\end{equation}
for the Skorokhod topology. For every $n\geq 1$, and every $r\geq 1$, set
$$\sigma^{(n)}_r=\min\{k: H^{(n)}_k\geq r\}.$$
Then it easily follows from \eqref{a.s.scalinghull} that
$$\Big(\frac{1}{n} \sigma^{(n)}_{[n^{1/3}t]}\Big)_{t\geq 0}\xrightarrow[n\to\infty]{a.s.} (\eta_{t/\mathsf{h}_{\two}})_{t\geq 0},$$
uniformly on every compact time set. By combining the latter convergence with \eqref{a.s.scalinghull} we arrive at
the a.s. convergence in the Skorokhod sense,
$$\Big(n^{-2/3} P^{(n)}_{\sigma^{(n)}_{[n^{1/3}t]}}, n^{-4/3} V^{(n)}_{\sigma^{(n)}_{[n^{1/3}t]}}\Big)_{t\geq 0}
\xrightarrow[n\to\infty]{a.s.} \Big(\mathsf{p}_{\two} S^+_{\eta_{t/\mathsf{h}_{\two}}}, \mathsf{v}_{\two} Z_{\eta_{t/\mathsf{h}_{\two}}}\Big)_{t\geq 0}.
$$
Recalling the identity in distribution \eqref{ident-Lamperti}, we get the
convergence in distribution of Theorem \ref{thm:scalinghull} since 
$$(P^{(n)}_{\sigma^{(n)}_r}, V^{(n)}_{\sigma^{(n)}_r})_{r\geq 0} \build{=}_{}^{(d)}  (P_{\sigma_r}, V_{\sigma_r})_{r\geq 0}=
(|\partial B^\bullet_r|,|B^\bullet_r|)_{r\geq 0}.$$
This completes the proof. \endproof

\section{Application to other distances}

In this section, we apply our techniques to study other distances on the UIPT (in fact on the dual graph of the UIPT) in order to get similar results for the scaling limits of the associated hull processes. Specifically, we will consider the dual graph distance and the first-passage percolation distance with exponential edge weights on the dual graph.
\subsection{The dual graph distance}

\label{sec:layers-dual}

We consider the dual map 
of the UIPT, whose vertices are in one-to-one correspondence with the faces of the UIPT, and each 
edge $e$ of the UIPT corresponds to an edge of the dual map between
the two faces incident to $e$. This dual map is denoted by $T_{\infty}^*$. By convention, the root vertex of $T_{\infty}^*$ 
or root face is the face incident to
the right-hand side of the root edge of the UIPT. We denote the graph distance on $T_\infty^*$  or dual graph distance
by $ \mathrm{d}_{ \mathrm{gr}}^*$. For every integer $r\geq 0$, we let $B_{r}^{\bullet,*}(T_{\infty})$ denote the hull of the ball of radius $r$ for $ \mathrm{d}_{ \mathrm{gr}}^*$.
This is the union of all faces of $T_{\infty}$ that are at dual graph distance smaller than or equal to $r$ from the root face, together with the finite regions these faces may enclose.

Similarly as in the previous section we now design a peeling algorithm which discovers these dual hulls step by step. In the first step ($n=0$) we reveal the root face. In the second step ($n=1$), we peel any edge incident to the root face. Then
inductively at step $n+1$ we peel the edge of the boundary of  $ \mathsf{T}_{n}$ which lies immediately on the right of the last revealed triangle (but not incident to that triangle). See Fig.~\ref{fig:layers-dual} for an illustration.\\

\begin{figure}[!h]
 \begin{center}
 \includegraphics[width=16cm]{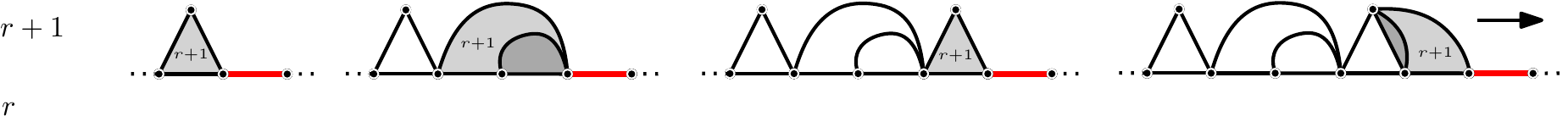}
 \caption{Illustration of the peeling by layers on the dual map. When $B^{\bullet,*}_{r}(T_{\infty})$ has been discovered, we turn around the boundary $\partial B_{r}^{\bullet,*}(T_{\infty})$ from left to right in order to reveal the next layer and obtain $B_{r+1}^{\bullet,*}(T_{\infty})$. \label{fig:layers-dual}}
 \end{center}
 \end{figure}
\vspace{-5mm}
As in the case of the peeling by layers for the graph distance on the primal lattice, one can prove by induction that, for every $n\geq 0$, there is an integer $h\geq 0$ such that one and only one of the following two possibilities occurs. Either all faces incident to $\partial \mathsf{T}_{n}$ are at the same dual graph distance $h$ from the root face of the UIPT. Or  $\partial \mathsf{T}_{n}$ contains both edges incident to faces at dual distance $h$ and 
edges incident to faces at dual distance $h+1$ from the root face. In the last case, these edges form two connected subsets of the boundary and  the edge that will be ``peeled off'' at step $n+1$ is the only edge incident to a face in $ \mathsf{T}_{n}$ at dual distance $h$ such that the edge immediately on its left is incident to a face of $ \mathsf{T}_{n}$ at dual distance $h+1$. In both cases we write $H^*_n=h$. As in the previous sections, we let $P_{n}$ and $V_{n}$ stand respectively for the perimeter and for the volume of the triangulation discovered after $n$ peeling steps.

\begin{proposition}[Distances in the peeling by layers on the dual map]  \label{prop:scalinglayersdual} We have the following convergence in distribution for the Skorokhod topology 
 \begin{eqnarray*} \left(  \frac{P_{[nt]}}{ \mathsf{p}_{\two} \cdot n^{2/3}}, \frac{V_{[nt]}}{  \mathsf{v}_{\two} \cdot  n^{4/3}}, \frac{H^*_{[nt]}}{ \mathsf{h}^*_{\two}\cdot n^{1/3}}\right)_{t \geq 0}  &\xrightarrow[n\to\infty]{(d)}&  \left(S_{t}^+, Z_{t},\int_{0}^t \frac{ \mathrm{d}u}{S_{u}^+} \right)_{t \geq 0},  \end{eqnarray*}
 where $ \mathsf{h}^*_{\two}=   ({1+  \mathsf{a}_{\two}})/{ \mathsf{p}_{\two}} = (16/3)^{1/3}$.
\end{proposition}
Theorem \ref{thm:scalinghulldual} is derived from Proposition \ref{prop:scalinglayersdual} in exactly 
the same way as Theorem \ref{thm:scalinghull} is derived from Proposition \ref{prop:scalinglayers} in Section \ref{sec:lamperti}. 
Let us briefly discuss the proof of Proposition \ref{prop:scalinglayersdual}, which follows the same lines as that of Proposition \ref{prop:scalinglayers}. The convergence of the first two components is again a consequence of Theorem \ref{thm:scalingpeeling}, and we focus on the convergence of the third component.  As for the peeling by layers on the primal lattice, the key idea is to consider the speed at which the peeling by layers (on the dual 
map) ``turns'' around the boundary. More precisely we denote the set of all edges of $T_{\infty}$ that are part of $ B_{r}^{\bullet,*}(T_{\infty})$ for some $r \geq 0$ by $ \mathcal{L}^*$, and we let $A_{n}^*$
stand for the number of edges of $ \mathsf{T}_{n} \backslash \partial \mathsf{T}_{n}$ that belong to $ \mathcal{L}^*$. We aim at the following analog of Proposition \ref{prop:A->1/3}:
  \begin{eqnarray}   
  \label{analoProp13}
  \frac{A_{n}^*}{n} \xrightarrow[n\to\infty]{(P)}  \mathsf{a}_{\two} +1 = 4/3.\end{eqnarray}
The idea to prove this convergence is the same as before: For most times $n$, the boundary $\partial \mathsf{T}_{n} $ has both a large number of
edges incident to a face of $\mathsf{T}_{n}$ at dual distance $H^*_{n}+1$ from the root face, and  a large number of
edges incident to a face of $\mathsf{T}_{n}$ at dual distance $H^*_n$. Then, except on a set of small probability, the only events  leading to a nonzero value of $\Delta A^*_n$ are events of type $\mathsf{R}_k$, for which  \begin{eqnarray*} \Delta A_n=-\Delta P_n +1= k+1.  \end{eqnarray*}
Note the additional  term $+1$ in the last display (compare with \eqref{eq:heuristicdeltaAn}) coming from the fact that we peel an edge belonging to $ \mathcal{L}^*$ at every step. This 
additional term explains why we get the limit $\mathsf{a}_{\two} +1$ in \eqref{analoProp13}, instead of $\mathsf{a}_{\two}$ in Proposition \ref{prop:A->1/3}. Apart from this difference, the technical 
details of the proof of \eqref{analoProp13} are very similar to those of Proposition \ref{prop:A->1/3}. For the analog of 
Lemma \ref{ideal-envir}, we introduce the number $U_{n}^*$ of edges of $ \partial \mathsf{T}_{n}$ that are incident to a face of $\mathsf{T}_{n}$ at dual distance $H_{n}^*$ from the root face, and $G_{n}^* = P_{n}-U_{n}^*$. Then $(P_{n},G^*_{n},H^*_{n})_{n \geq 0}$  is a Markov chain 
taking values
in $\{(p,\ell,h)\in\Z^3: p\geq 2, 0\leq \ell\leq p-1, h\geq 0\}$, whose transition
kernel $Q^*$ is specified as follows:
\begin{equation}
\label{transi-dual}
\begin{array}{ll}
Q^*((p,\ell,h),(p+1,\ell+2,h))= q^{(p)}_{-1}\qquad&\hbox{if }\ell\leq p-2\\
Q^*((p,p-1,h),(p+1,0,h+1))= q^{(p)}_{-1}\qquad&\\
Q^*((p,\ell,h),(p-k,\ell-k+1,h))= q^{(p)}_k\qquad&\hbox{for }1\leq k\leq \ell\\
Q^*((p,\ell,h),(p-k,\ell +1 ,h))= q^{(p)}_k\qquad&\hbox{for }1\leq k\leq p-\ell-2\\
Q^*((p,\ell,h),(p-k,1,h))= q^{(p)}_k\qquad&\hbox{for }\ell+1\leq k\leq p-2\\
Q^*((p,\ell,h),(p-k,0,h+1))= q^{(p)}_k\qquad&\hbox{for }p-\ell-1\leq k\leq p-2\,.
\end{array}
\end{equation}
The analog of Lemma \ref{ideal-envir} then holds with $G_i$ and $U_i$ replaced respectively by $G^*_i$ and $P^*_i$,
with a very similar proof. This provides the key technical ingredient needed to adapt the proof of Proposition \ref{prop:A->1/3}
in order to get the convergence \eqref{analoProp13}. Finally, an analog of Lemma \ref{lem:H->0} also holds with
$\E[H_n]$ replaced by $\E[H_n^*]$, and Proposition \ref{prop:scalinglayersdual} can then be derived from
\eqref{analoProp13} in the same way as Proposition \ref{prop:scalinglayers} was derived from 
Proposition \ref{prop:A->1/3} in Section \ref{sec:dist-layers}. We leave the details to the reader. 

\subsection{First-passage percolation}

\label{sec:fpp}

We now assign independent weights exponentially distributed with parameter $1$ to the edges of $T^*_\infty$. The weight of a path in $T^*_\infty$ is just the
sum of the weights of its edges. We let $\mathsf{F}_0$ consist only of the root 
face and, 
for every $t >0$, we let $ \mathsf{F}_{t}$ be the union of all faces of the UIPT 
which are connected to the root face by a dual path whose weight is less than or equal to $t$. We then let
$\mathsf{F}^\bullet_{t}$ be the hull of $ \mathsf{F}_{t}$. We set $\tau_0=0$ and we let $0<\tau_{1} < \tau_2<\cdots < \tau_{n} <\cdots$ be the successive jump times of the process $ { \mathsf F}^\bullet_t$ (a simple argument shows that $\tau_n\to\infty$ as 
$n\to\infty$, which will also follow from the next proposition). Note
that, at each time $\tau_n$ with $n\geq 1$, a new triangle
incident to the boundary of $\mathsf{F}^\bullet_{\tau_{n-1}}$ is added to
 $ { \mathsf F}^\bullet_{\tau_{n-1}}$, together with the triangles in the 
 ``hole'' that this addition may create.

By convention we
let $\tilde{\mathsf{F}}_0$ be the trivial triangulation, and we set, for every $n \geq 1$ 
$$ \tilde{ \mathsf{F}}_{n} =  \mathsf{F}^\bullet_{\tau_{n-1}}.$$ The following
proposition shows that the process $(\tilde{\mathsf F}_{n})_{n\geq 0}$
is a particular instance of a peeling process, which is called the
 \emph{uniform} peeling process or Eden model on the UIPT.
 See also \cite[Section 6]{AB14}.
 
\begin{proposition}  \label{prop:fpppeeling} The sequence $ (\tilde{ \mathsf{F}}_{n})_{n \geq 0}$ has the same law as the sequence $(\mathsf{T}_n)_{n\geq 0}$ corresponding to a peeling process where at step $1$ we reveal the
triangle incident to the right-hand side of the root edge, and for every 
$n\geq 2$,  conditionally on
$\mathsf{T_0},\ldots,\mathsf{T}_{n-1}$, the peeled edge at step $n$ is chosen  uniformly at random among the edges
of $\partial\mathsf{T}_{n-1}$.
Furthermore, conditionally on the sequence $(\tilde{\mathsf{F}}_n)_{n\geq 1}$,
the increments $\tau_1-\tau_0,\tau_2-\tau_1,\tau_3-\tau_2,\ldots$ are
independent, and, for every $k\geq 1$, 
$\tau_{k}-\tau_{k-1}$ is exponentially distributed with parameter $ | \partial \tilde{\mathsf{F}}_{k}|$.
\end{proposition}

 \noindent{\bf Remark.} Since $ | \partial \tilde{\mathsf{F}}_{k}|\leq k+2$, the 
 last assertion shows that $\tau_k\uparrow \infty$ a.s. as $k\to\infty$.

\proof 
Let $n\geq 1$. Consider an edge $e$ of  $\partial\tilde{\mathsf F}_n$. Then, $e$ is incident
to a unique face $f_e$ of ${\mathsf F}_{\tau_{n-1}}$, and we write 
$ \mathrm{d_{fpp}}(f_e)$ for the first-passage percolation distance
between $f_e$ and the root face (in other words, this is the 
minimal weight of a dual path connecting the root face and $f_e$).
We also write $w_e$ for the weight of $e$, or rather of its dual edge. 
Since $f_e$ is contained in ${\mathsf F}_{\tau_{n-1}}$, we have 
$\mathrm{d_{fpp}}(f_e)\leq \tau_{n-1}$, with equality only if $f_e$
is the triangle that was added at time $\tau_{n-1}$. 
Also it is clear that
$$w_e > \tau_{n-1} - \mathrm{d_{fpp}}(f_e)$$
because otherwise this would contradict the fact that the other face incident to
$e$ is not in ${\mathsf F}_{\tau_{n-1}}$.

Next the lack of
memory of the exponential distribution ensures that,
conditionally on the variables $(\tilde{\mathsf F}_0, \tilde{\mathsf F}_1,\ldots,\tilde{\mathsf F}_n,\tau_1,\ldots,\tau_{n-1})$, the random variables 
$$w_e - (\tau_{n-1} - \mathrm{d_{fpp}}(f_e)),$$
where $e$ varies over the edges of $\partial\tilde{\mathsf F}_n$, are
independent and exponentially distributed with parameter $1$. Now observe
that the next jump  will occur at time 
$$\tau_n = \tau_{n-1} + \min\{w_e - (\tau_{n-1} - \mathrm{d_{fpp}}(f_e)): e\hbox{ edge of }\partial\tilde{\mathsf F}_n\}.$$
It follows that, conditionally on $(\tilde{\mathsf F}_0, \tilde{\mathsf F}_1,\ldots,\tilde{\mathsf F}_n,\tau_1,\ldots,\tau_{n-1})$, the variable $\tau_n-\tau_{n-1}$ is exponential
with parameter $ | \partial \tilde{\mathsf{F}}_{n}|$, and furthermore, the new triangle
added to $\tilde{\mathsf{F}}_{n}$ corresponds to the edge attaining 
the preceding minimum, which is therefore uniformly
distributed over edges of $\partial \tilde{\mathsf{F}}_{n}$. This completes the proof.
\endproof

\begin{proof}[Proof of Theorem \ref{thm:scalingfpp}] 
As in the previous sections, we use the notation $V_{n}$ and $P_{n}$ for the volume and the perimeter of $ \tilde{ \mathsf{F}}_{n}$.
We will establish the following convergence in distribution for the Skorokhod topology 
 \begin{eqnarray} \label{eq:goalfpp} \left(  \frac{P_{[nt]}}{ \mathsf{p}_{\two} \cdot n^{2/3}}, \frac{V_{[nt]}}{  \mathsf{v}_{\two} \cdot  n^{4/3}}, \frac{\tau_{[nt]}}{ (1/ \mathsf{p}_{\two})\cdot n^{1/3}}\right)_{t \geq 0}  &\xrightarrow[n\to\infty]{(d)}&  \left(S_{t}^+, Z_{t},\int_{0}^t \frac{ \mathrm{d}u}{S_{u}^+} \right)_{t \geq 0}.  \end{eqnarray}
Theorem \ref{thm:scalingfpp} then follows from 
\eqref{eq:goalfpp} by the very same arguments we used 
to deduce Theorem \ref{thm:scalinghull} from Proposition \ref{prop:scalinglayers}
in Section \ref{sec:lamperti}.  

The joint convergence of the first two components  in \eqref{eq:goalfpp}  is given by Theorem  \ref{thm:scalingpeeling}. So we need to prove the convergence of the
third component and to check that it holds jointly
with the first two. As in the proof of Proposition \ref{prop:scalinglayers},
we fix $0<s<t$ and we first consider $\tau_{[nt]}-\tau_{[ns]}$. Writing
$$\tau_{[nt]}-\tau_{[ns]}= \sum_{i=[ns]+1}^{[nt]} (\tau_i-\tau_{i-1})$$
and, using Proposition \ref{prop:fpppeeling}, we see that, conditionally
on $(P_k)_{k\geq 0}$, the variable $\tau_{[nt]}-\tau_{[ns]}$ is distributed as
$$\sum_{i=[ns]+1}^{[nt]} \frac{\mathbf{e}_i}{P_i},$$
where the random variables $\mathbf{e}_1,\mathbf{e}_2,\ldots$ are
independent and exponentially distributed with parameter $1$, and
are also independent of $(P_k)_{k\geq 0}$. By the convergence of the first component in 
\eqref{eq:goalfpp}, we have
$$n^{-1/3}\sum_{i=[ns]+1}^{[nt]} \frac{1}{P_i}
= \int_{n^{-1}([ns]+1)}^{n^{-1}([nt]+1)} \frac{\mathrm{d}u}{n^{-2/3}P_{[nu]}}
\xrightarrow[n\to\infty]{(d)}
\frac{1}{\mathsf{p}_{\two}} \int_s^t \frac{\mathrm{d}u}{S^+_u},$$
and, on the other hand,
\begin{align*}
&E\Bigg[\Big(n^{-1/3}\sum_{i=[ns]+1}^{[nt]} \frac{\mathbf{e}_i}{P_i} \,-\,
n^{-1/3}\sum_{i=[ns]+1}^{[nt]} \frac{1}{P_i}\Big)^2\,\Bigg|\, (P_k)_{k\geq 0}\Bigg]\\
&\quad= n^{-2/3} \sum_{i=[ns]+1}^{[nt]} \frac{1}{(P_i)^2}
= \frac{1}{n}\int_{n^{-1}([ns]+1)}^{n^{-1}([nt]+1)} \frac{\mathrm{d}u}{(n^{-2/3}P_{[nu]})^2}
\end{align*}
converges to $0$ in probability as $n\to\infty$. It easily follows that 
\begin{equation}
\label{fpptech}
n^{-1/3}\,(\tau_{[nt]}-\tau_{[ns]}) \xrightarrow[n\to\infty]{(d)}
\frac{1}{\mathsf{p}_{\two}} \int_s^t \frac{\mathrm{d}u}{S^+_u},
\end{equation}
and the previous argument also shows that this convergence
holds jointly with that of the first two components in \eqref{eq:goalfpp}. We can complete the
proof by arguing in a way similar to the end of the proof of Proposition \ref{prop:scalinglayers}.
It suffices to verify that
$$\sup_{n\geq 1}\,\E[n^{-1/3}\tau_{[ns]}] \xrightarrow[s\to 0]{} 0.$$
This is however very easy, since 
$$\E[\tau_{[ns]}]= \E\Big[\sum_{i=1}^{[ns]} \frac{1}{P_i}\Big]$$
and we can use Lemma \ref{lem:1/P} to obtain that $\E[\tau_{[ns]}]\leq C\,(ns)^{1/3}$,
for some constant $C$.
\end{proof}

\subsection{Comparing distances}

\label{rem:fpp} One conjectures that balls for the dual graph distance or the first-passage percolation distance grow asymptotically like ``deterministic'' balls for the graph distance. More precisely,
one expects that there exist two constants $c_{1},c_{2}>0$ such that, for every
$\ve>0$, one has
\begin{eqnarray*}
&&B_{(c_{1}^{-1}-\ve)r}\subset  {B}_r^* \subset B_{(c_{1}^{-1}+\ve)r},\\
&&B_{(c_{2}^{-1}-\ve)r}\subset  \mathsf{F}_r \subset B_{(c_{2}^{-1}+\ve)r}.
\end{eqnarray*}
with high probability when $r$ is large (here $B_{r}^*$ is the ball for the dual graph distance, that is the union of 
all faces that are at dual graph distance less than or equal to $r$ from the root face).
The reason for this belief is the fact that the UIPT is ``isotropic'', in contrast
with deterministic lattices such as $\mathbb{Z}^d$. 
Our results support the previous conjecture since the scaling limits for the
perimeter and volume of hulls are the same for any the balls $B_r$, $ \mathsf{F}_r$ and $B_{r}^*$,
up to multiplicative constants. Note that, if the last display holds, we must have also
$$B^{\bullet,*}_{(c_{1}-\ve)r}\subset  B^{\bullet}_r \subset B^{\bullet,*}_{(c_{1}+\ve)r},$$
$$ \mathsf{F}^\bullet_{(c_{2}-\ve)r}\subset  B^\bullet_r \subset  \mathsf{F}^\bullet_{(c_{2}+\ve)r},$$
with high probability when $r$ is large.
By comparing the limits in distribution of $|B^\bullet_r|$ (Theorem 	\ref{thm:scalinghull}),
 of $|B_{r}^{\bullet,*}|$ (Theorem \ref{thm:scalinghulldual}) and of $| \mathsf{F}_{r}^\bullet|$ (Theorem  \ref{thm:scalingfpp}), we see that if the previous conjecture
holds, the constants $c_{1}$ and $c_{2}$ must be equal to
$$c_{1}= \frac{ \mathsf{h}^*_{\two}}{ \mathsf{h}_{ \two}} = \frac{1 + \mathsf{a}_{\two}}{ \mathsf{a}_{\two}} = 4 \quad \mbox{ and } \quad c_{2}= \frac{1}{ \mathsf{p}_{\two} \mathsf{h}_{\two}}= \frac{1}{ \mathsf{a}_{\two}} = 3.$$

See \cite[Remark 5]{AB14} for related calculations about 
two-point and three-point functions  for first-passage percolation on
type I triangulations 
(in that case, the analog of the constant $ \mathsf{a}_{\two}$
is $ \mathsf{a}_{\one} = 1/(2 \sqrt{3})$, as we shall see below). The proof of the above conjecture is discussed, in the slightly different setting
of type I triangulations, in the forthcoming work \cite{CLGmodif}.

\section{Other models}
\label{sec:general}
Although we chose to focus on type II triangulations, our results can be extended to other classes of infinite random planar maps. Roughly speaking, one only needs to replace the constants $ \mathsf{a}_{\two}, \mathsf{t}_{\two}, \mathsf{b}_{\two}$ defined in \eqref{eq:tailX}, in Proposition \ref{prop:boltz} and in Proposition \ref{prop:A->1/3}  by their appropriate values in the model in consideration.  
All our results should then go through with the constants $ \mathsf{v}_{\cdot },\mathsf{p}_{\cdot}, \mathsf{h}_{\cdot}, \mathsf{h}^*_{\cdot}$ evaluated via the same 
``universal'' relations from the constants $ \mathsf{a}_{\cdot}, \mathsf{t}_{\cdot}, \mathsf{b}_{\cdot}$. In this section we carefully explain how to do this in two
particular cases, namely type I triangulations and quadrangulations. It may well be the case that our techniques can be extended to even more general classes of random planar maps such as the (regular critical) Boltzmann triangulations considered in \cite{Bud15}.

\subsection{Type I triangulations}
Let us consider the case of type I triangulations, where both loops and
multiple edges are allowed. The construction of the UIPT in this case is not treated by Angel and Schramm \cite{AS03},
but the techniques of \cite{AS03} can easily be extended using the corresponding
enumeration results (see below). Alternatively, the construction of the type I UIPT follows as 
a special case of the recent results of Stephenson \cite{St14}.
We denote the UIPT
for type I triangulations by $ T_{\infty}^{(1)}$. Let us list the enumeration results 
corresponding to those of Section \ref{sec:enumer}. These results
may be found in Krikun \cite{Kri07} (Krikun uses the number of edges as the size parameter and in order to apply his formulas we note that a triangulation of the $p$-gon 
with $n$ inner vertices has $3n+2p-3$ edges).

For every $p\geq 1$ and $n\geq 0$, let $ \mathcal{T}^{(1)}_{n,p}$ stand for the set of all type I triangulations with 
$n$ inner vertices and a simple boundary of length $p$, which are rooted
on an edge of the boundary in the way explained in 
Section \ref{sec:enumer}. We have for $(n,p) \ne (0,1)$
$$\# \mathcal{T}_{n,p}^{(1)} =4^{n - 1}  \frac{p\, (2p)!\,(2 p + 3 n - 5)!!}{(p!)^2\,n! \,
(2 p + n - 1)!!} \underset{n \to \infty}{\sim}   C^{(1)}(p) \,(12 \sqrt{3})^{n}\, n^{-5/2},$$
where
$$C^{(1)}({p}) = \frac{3^{p-2} \, p \, (2 p)!}{4 \sqrt{2 \pi} \,  (p!)^2}  \underset{p \to \infty}{\sim} \frac{1}{36\pi \sqrt{2}} \, \sqrt{p} \ 12^p. 
$$
We then set
$Z^{(1)}(p) = \sum_{ n \geq 0} \# \mathcal{T}_{n,p}^{(1)} (12 \sqrt{3})^{-n}$ and we have the formula (see \cite[Section 2.2]{ACpercopeel})
$$Z^{(1)}(p) = \frac{6^p\,(2p-5)!! }{8 \sqrt{3} \,p!} \quad \text{if \ } p\geq 2,
  \qquad \qquad  Z^{(1)}(1) =  \frac{2 - \sqrt{3}}{4}.  $$
The generating series of $Z^{1}(p)$ can also be computed explicitly from \cite[formula (4)]{Kri07} 
and an appropriate change of variables (we omit the details):
  \begin{eqnarray*}  \sum_{p\geq 0} Z^{(1)}(p+1) z^p=  \frac{1}{2}+ \frac{ (1- 12 z)^ {3/2}-1}{24 \sqrt {3} z}. \end{eqnarray*}
In particular, the analog of \eqref{eq:asympZp} is
$$Z^{(1)}(p+1) \underset{p\to \infty}{\sim}   \frac{ \sqrt{3}}{8 \sqrt{\pi}} \,12^p \,p^{-5/2},$$
and similarly as in \eqref{eq:asympZp}, we set
$$\mathsf{t}_\one = \frac{ \sqrt{3}}{8 \sqrt{\pi}}.$$

The peeling algorithm discovering $ T_{\infty}^{(1)}$ is then described in a very similar
way as in Section \ref{sec:peeling}. The only difference is that we now need 
to consider the possibility of loops. With the notation of Section \ref{sec:peeling}, 
and supposing that the revealed region has a boundary of size $p\geq 1$, events of type $ \mathsf{L}_{0}$ or $ \mathsf{R}_{0}$,
or  of type $ \mathsf{L}_{p-1}$ or $ \mathsf{R}_{p-1}$, may occur (the definition
of these events should be obvious from Fig.\,\ref{cases}). The respective probabilities
of events $ \mathsf{C}$, $ \mathsf{L}_{k}$ or $ \mathsf{R}_{k}$ are given by
formulas analogous to \eqref{probatypeC} and \eqref{probatypeR}, where 
$2/27$ is replaced by $1/(12\sqrt{3})$, the functions $C$ and $Z$  are replaced 
respectively by $C^{(1)}$ and $Z^{(1)}$, and
finally $k$ is allowed to vary in $\{0,\ldots,p-1\}$. 

An analog of Proposition \ref{prop:positive} holds, and the constant
$\mathsf{p}_\two$ has to be replaced by
$$\mathsf{p}_\one = \Big( \frac{8\mathsf{t}_\one\sqrt{\pi}}{3}\Big)^{2/3} = 3^{-1/3}.$$
Similarly, there is a version of Proposition \ref{prop:boltz} in the type I case, and the
constant $\mathsf{b}_\two$ is replaced by 
$$\mathsf{b}_\one= \frac{4}{3}$$
whereas the limiting distribution remains the same. Finally, the analog of Proposition 
\ref{prop:A->1/3} involves the new constant
$$\mathsf{a}_\one = \frac{1}{2\sqrt{3}}.$$

The proofs of Theorems \ref{thm:scalingpeeling}, \ref{thm:scalinghull}, \ref{thm:scalinghulldual} and \ref{thm:scalingfpp} can then be adapted easily to the UIPT $ T_{\infty}^{(1)}$. In these
statements,  $\mathsf{p}_\two$ is replaced by
$\mathsf{p}_\one$ and the other constants 
$ \mathsf{v}_{\two}, \mathsf{h}_{\two}$ and $ \mathsf{h}^*_{\two}$ are replaced respectively by 
$$\mathsf{v}_{\one} = (\mathsf{p}_\one)^2 \mathsf{b}_\one= 4\cdot 3^{-5/3}
,\qquad \mathsf{h}_{\one} =\frac{\mathsf{a}_\one}{\mathsf{p}_\one}=  \frac{1}{2 }\,3^{-1/6} \quad \hbox{and}\quad \mathsf{h}^*_{\one} =\frac{1+\mathsf{a}_\one}{\mathsf{p}_\one}$$ 
We note that $\mathsf{p}_\one/(\mathsf{h}_\one)^2= \mathsf{p}_\two/(\mathsf{h}_\two)^2$,
which, by Theorem \ref{thm:scalinghull} and its type I analog, means that the scaling 
limit of the perimeter of hulls is exactly the same for type I and for type II triangulations.
This fact can be explained by a direct relation between the UIPTs of type I and of type II, but
we omit the details. 
 
\subsection{Quadrangulations}

Let us now consider the Uniform Infinite Planar Quadrangulation (UIPQ), which is
denoted here by $Q_\infty$.
This case requires more changes in the arguments. We first note that a quadrangulation with a simple boundary necessarily has an even perimeter. 
For every $p\geq 1$, 
let $ \mathcal{Q}_{n,p}$ stand for  the set of all quadrangulations with a simple boundary of perimeter $2p$  and $n$ inner vertices, which are rooted at an oriented edge of the boundary in such a way that the external face lies on the right
of the root edge.  For $n\geq 0$ and $p \geq1$, we read from \cite[Eq.~(2.11)]{BG09}  that
$$
  \#  \mathcal{Q}_{n,p} = 3^{n-1} \frac{(3 p)! (3p-3 +2n)!}{
n! p! ( 2 p-1)! (n+3p-1)!} \underset{n\to\infty}{\sim}  C^\square({p}) 12^n n^{-5/2}, 
$$
where
$$C^\square(p) = \frac{8^{p-1} (3 p)!}{3 \sqrt{\pi} p! (2p-1)!} \underset{p\to\infty}{\sim} \frac{1}{8
\sqrt{3} \pi} 54^p \sqrt{p}.$$
We have also, for every $p\geq 2$,
  \begin{eqnarray*} Z^\square(p)=\sum_{n\geq 0}
\#  \mathcal{Q}_{n,p} 12^{-n} = \frac{8^p (3p-4)!}{(p-2)! (2p)!}, 
\end{eqnarray*}
and $Z^\square(1)=4/3$. Furthermore, 
  \begin{eqnarray} Z^\square(p+1) \underset{p\to\infty}{\sim} \frac{1}{ \sqrt{3\pi}} 54^{p}p^{-5/2}, \quad 
\sum_{k\geq 0} Z^\square(k+1)54^{-k} = 3/2 , \quad 
\sum_{k\geq 0} k Z^\square(k+1)54^{-k} = 1/2.   \end{eqnarray}

The transitions in the peeling process of the UIPQ are more complicated than previously because of 
additional cases. 
If at step $n \geq 0$ the perimeter of the discovered quadrangulation $ \mathsf{Q}_{n}$ 
is equal to $2m$,  then the revealed quadrangle at the next step may have three different shapes (see Fig.\,\ref{cases}): 
\begin{enumerate}
 \item  \textbf{Shape $ \mathsf{C}$:} The revealed  quadrangle has two vertices in the unknown region,  an event of probability $$\mathbb{P}( \mathsf{C} \mid | \partial  \mathsf{Q}_{n}| = 2m)=  \mathbf{q}_{-2}^{(m)} = 12^{-2}\frac{C^\square(m+1)}{C^\square(m)}.$$
 
 \item   \textbf{Shapes $ \mathsf{L}_{k}$ and $ \mathsf{R}_{k}$, for $k \in \{ 0,1, \ldots , 2m-1\}$:} The revealed  quadrangle has three vertices on the boundary of $ \mathsf{Q}_{n}$. This quadrangle then ``swallows'' a part of the boundary of $\partial \mathsf{Q}_n$ of length $k$. 
 This event is denoted by $ \mathsf{L}_{k}$ or $ \mathsf{R}_{k}$ according to whether the part of the boundary that is swallowed is on the right or on the left of the peeled edge. Note that
 the revealed face encloses a finite quadrangulation of perimeter $k+1$ if $k$ is odd and $k+2$ if $k$ is even. These events have probability  
 \begin{align*} \mathbb{P}(  \mathsf{L}_{2k} \mid | \partial \mathsf{Q}_{n}| = 2m) &= \mathbb{P}(  \mathsf{L}_{2k+1} \mid | \partial \mathsf{Q}_{n}| = 2m)  = \mathbb{P}(  \mathsf{R}_{2k} \mid | \partial \mathsf{Q}_{n}| = 2m) = \mathbb{P}(  \mathsf{R}_{2k+1} \mid | \partial \mathsf{Q}_{n}| = 2m)\\ 
 &=  \mathbf{q}^{(m)}_{2k+1} = \mathbf{q}^{(m)}_{2k} = \frac{Z^\square(k+1)}{12}\frac{C^\square(m-k)}{C^\square(m)}.  \end{align*}
 
 \item \textbf{Shapes $ \mathsf{L}_{k_{1},k_{2}}$, $ \mathsf{R}_{k_{1},k_{2}}$ and $ \mathsf{C}_{k_{1},k_{2}}$
 for $k_1,k_2\geq 1$ odd and such that $k_1+k_2<2m$:} This last case occurs when the revealed quadrangle has its four vertices on $ \partial \mathsf{Q}_{n}$. It then encloses two finite quadrangulations of respective perimeters $k_{1}+1$ and $k_{2}+1$ either both on the left side of the peeled edge in case $ \mathsf{L}_{k_{1},k_{2}}$, or one
 on each side of the peeled edge  in case $ \mathsf{C}_{k_{1},k_{2}}$, or both on the right side of the peeled edge in case $ \mathsf{R}_{k_{1},k_{2}}$. These three events have the same probability: writing $k_1=2j_1+1$ and $k_2=2j_2+1$, with $j_1+j_2<m-1$,
  \begin{align*} \mathbb{P}( \mathsf{L}_{k_{1},k_{2}} \mid | \partial \mathsf{Q}_{n}| =2m)&=\mathbb{P}( \mathsf{R}_{k_{1},k_{2}} \mid | \partial \mathsf{Q}_{n}| =2m)=\mathbb{P}( \mathsf{C}_{k_{1},k_{2}} \mid | \partial \mathsf{Q}_{n}| =2m)\\
 &=\mathbf{q}^{(m)}_{k_{1},k_{2}} = Z^\square(j_{1}+1)Z^\square(j_{2}+1)\frac{C^\square(m-j_{1}-j_{2}-1)}{C^\square(m)}.  \end{align*}
\end{enumerate}
\begin{figure}[!h]
 \begin{center}
 \includegraphics[width=16cm]{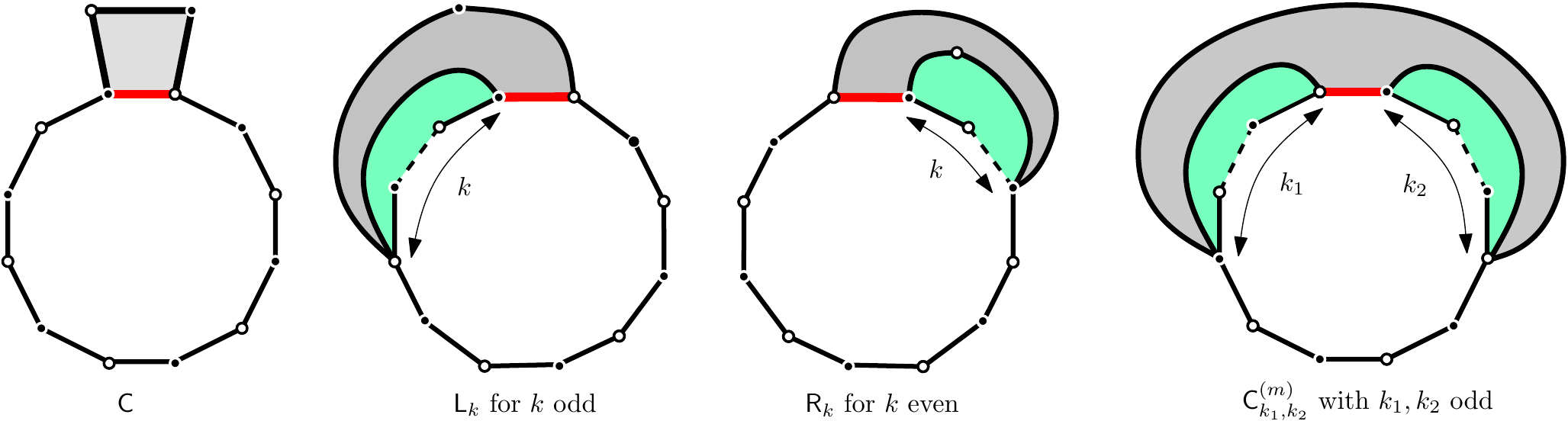}
 \caption{A few peeling transitions in the quadrangular case.}
 \end{center}
 \end{figure}

Furthermore, conditionally on each of the above cases, the finite quadrangulations enclosed by the revealed face are independent Boltzmann quadrangulations with the prescribed perimeters. Let $ P^\square_{n}$ stand for the half-perimeter at step $n$ in the peeling process.
Then, similarly as in the triangular case, the Markov chain $(P^\square_{n})$ 
is obtained by conditioning a random walk $X$ on $\mathbb{Z}$   to stay (strictly) positive,
and the increments of $X$ are now distributed as follows:
 \begin{eqnarray*}\mathbb{E}[f(X_{n+1})\mid X_n] = f(X_{n}+1) \cdot  \mathbf{q}_{-2} +
 \sum_{k=0}^\infty f(X_{n}-k) \cdot \Big(2( \mathbf{q}_{2k}+ \mathbf{q}_{2k+1}) + 3 \sum_{\begin{subarray}{c} k_{1}+k_{2} = 2k\\ k_{1},k_{2} \geq 1 \ \mathrm{ odd} \end{subarray}}  \mathbf{q}_{k_{1},k_{2}}\Big),  \end{eqnarray*}
 where $ \mathbf{q}_{j}= \lim_{m\to \infty}  \mathbf{q}_{j}^{(m)}$ and 
 $ \mathbf{q}_{k_1,k_2}= \lim_{m\to \infty}  \mathbf{q}_{k_1,k_2}^{(m)}$ as in the triangular case. From the enumeration results, we get, for every $k\geq 0$,
  \begin{eqnarray} \label{eq:tailXq}  \mathbb{P}( \Delta X  = -k) = 2(q_{2k}+q_{2k+1}) + 3 \sum_{\begin{subarray}{c} k_{1}+k_{2} = 2k\\ k_{1},k_{2} \geq 1 \ \mathrm{ odd} \end{subarray}} q_{k_{1},k_{2} } \underset{k\to\infty}{\sim}   \frac{1}{2 \sqrt{3 \pi}} k^{-5/2}.  \end{eqnarray}
  
The results of Sections \ref{sec:perimeter} and \ref{sec:volume} 
can then be extended to the UIPQ $Q_\infty$. Comparing \eqref{eq:tailXq} with \eqref{eq:tailX}, we see that the role of the constant $\mathsf{t}_\two$
is now played by $\mathsf{t}_\square= 1/(4\sqrt{3\pi})$.
Then the convergence in distribution
of Proposition \ref{prop:positive} holds for $ P^\square_{n}$, with the constant 
$\mathsf{p}_\two$ replaced by
$$\mathsf{p}_\square=\Big( \frac{8\mathsf{t}_\square\sqrt{\pi}}{3}\Big)^{2/3} = \frac{2^{2/3}}{3}.$$
An analog of Proposition \ref{prop:boltz}, where we now consider a 
Boltzmann quadrangulation $Q^{(p)}$ of the $2p$-gon, also holds in the form
$$p^{-2}\,\mathbb{E}[|Q^{(p)}|] \underset{p\to\infty}{\longrightarrow} \frac{9}{2}=: \mathsf{b}_\square.$$

The peeling by layers requires certain modifications in the case of quadrangulations. As previously, the ball $B_{r}(Q_{\infty})$ is
the planar map obtained by keeping only those faces of $Q_{\infty}$ that are incident to at least one vertex whose graph distance
from the root vertex is smaller than or equal to $r-1$, and the hull $ B_{r}^\bullet(Q_{\infty})$ is obtained by filling in the finite holes of $B_{r}(Q_{\infty})$.
The boundary $ \partial B_{r}^\bullet(Q_{\infty})$ is now a simple cycle that visits  alternatively  vertices at distance $r$ and $r+1$ from the root vertex. If we move around the boundary of this cycle in clockwise order, we encounter two types of (oriented) edges, edges $r+1\to r$ connecting a vertex at distance $r+1$ to a vertex at distance $r$, and edges $r\to r+1$ connecting a vertex at distance $r$ to a vertex at distance $r+1$.

To describe the  peeling by layers algorithm, suppose that, at a certain step of the peeling process, 
the revealed region is the hull $ {B}^\bullet_{r}(Q_{\infty})$. Then we choose  (deterministically or using some 
independent randomization) an edge of the boundary of type $r+1\to r$. We reveal the face incident to this edge that is not already in $ {B}^\bullet_{r}(Q_{\infty})$ and as usual we fill in the holes that may
have been created. At the next step, either the (new) boundary has an edge of type $r+1\to r$ that is incident to the quadrangle 
revealed in the previous step, and we peel this edge, or we peel the first edge of type $r+1\to r$ coming after the revealed quadrangle in clockwise order.
We continue inductively, ``moving around the boundary in clockwise order''. See Fig.~\ref{fig:layersq} for an example. After a finite number of steps, the boundary does not contain any vertex at distance $r$, and it is easy to verify that the revealed region is then the hull $ {B}^\bullet_{r+1}(Q_{\infty})$, so that we
can continue the construction by induction. 

\begin{figure}[!h]
 \begin{center}
 \includegraphics[width=14cm]{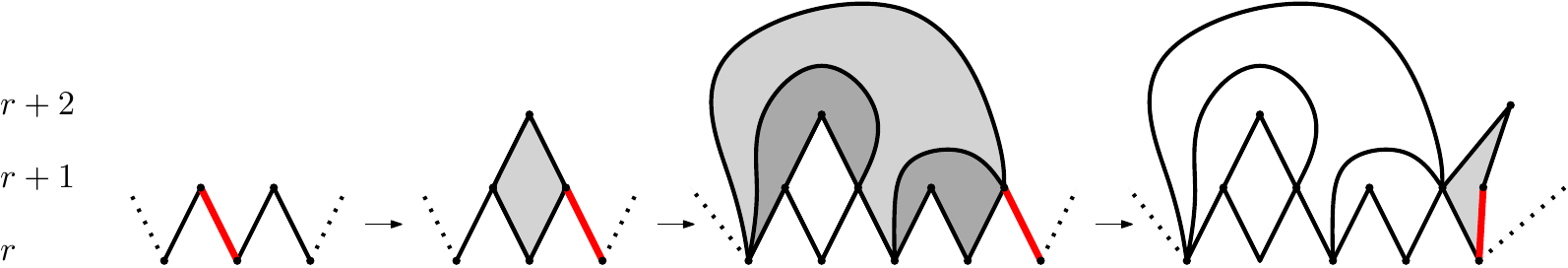}
 \caption{ \label{fig:layersq}Illustration of the peeling by layers in the quadrangular case: we choose an edge to start discovering the new layer and then peel from left to right all the edges that contain a vertex at distance $r$ from the root vertex.}
 \end{center}
 \end{figure}
Proposition \ref{prop:A->1/3} is adapted as follows. For every $r \geq 1$, let $ \mathscr{L}^\square_{r}$ be the set of all vertices of $ \partial B_{r}^\bullet(Q_{\infty})$ that are at distance exactly $r$ from the root vertex. 
Clearly the perimeter $ |\partial B_{r}^\bullet(Q_{\infty})|$ is equal to $2 \# \mathscr{L}^\square_{r}$. We also denote  the  union of all $ \mathscr{L}^\square_{r}$ for $r \geq 1$ by $ \mathscr{L}^\square$. Finally, for $n \geq 1$, we let $A^\square_{n}$ be the number of vertices of $ \mathscr{L}^\square$ that are in the interior of the discovered region at step $n$. Then the analog of Proposition \ref{prop:A->1/3} reads
$$ \frac{A_{n}}{n} \xrightarrow[n\to\infty]{(P)} \frac{1}{3} =:  \mathsf{a}_{\square} .$$
The idea of the proof is the same but technicalities become 
somewhat more complicated (we omit the details). 

Versions of Theorems \ref{thm:scalingpeeling}, \ref{thm:scalinghull} then hold for the UIPQ $Q_{\infty}$. In these
statements we now interpret the size of the boundary
as half its perimeter,  the constant $\mathsf{p}_\two$ is replaced by
$\mathsf{p}_\square$ and the other constants 
$ \mathsf{v}_{\two}$ and $ \mathsf{h}_{\two}$ are replaced respectively by 
$$\mathsf{v}_{\square} = (\mathsf{p}_\square)^2 \mathsf{b}_\square= 2^{1/3}
\quad\hbox{and}\quad \mathsf{h}_{\square} =\frac{\mathsf{a}_\square}{\mathsf{p}_\square}= 
2^{-2/3} .$$ 
 We observe that the convergence of volumes
in the analog of Theorem \ref{thm:scalinghull} for the UIPQ was already obtained in 
\cite{CLGHull} as a consequence of the invariance principles relating the UIPQ and
the Brownian plane (see Theorems 5.1 and 1.3 in \cite{CLGHull}). It would be 
significantly harder to derive the convergence of boundary lengths from
the same invariance principles. On the other hand, Krikun \cite{Kri05} has 
a version of the scaling limit for boundary lengths in the case of quadrangulations, but with
a different definition of hull boundaries leading to different constants. \medskip 

It is also possible to adapt Theorems \ref{thm:scalinghulldual} and \ref{thm:scalingfpp} to the setting of quadrangulations: the limiting process in the analog of 
Theorem \ref{thm:scalinghulldual} (where we again consider the half-perimeter rather than  the perimeter) is $ (\mathsf{p}_\square\cdot\mathcal{L}_{  t/ \mathsf{h}^*_{\square}   }, \mathsf{v}_\square\cdot\mathcal{M}_{t/\mathsf{h}^*_{\square} })_{t \geq 0}$, with 
$$ \mathsf{h}^*_{\square} =   \frac{1+ \mathsf{a}^*_{\square}}{2  \mathsf{p}_{\square}},$$ where $  \mathsf{a}^*_{\square}= \frac{1}{2}$ is the mean number of edges ``swallowed'' on the right of the peeled edge in 
a peeling step for the half-plane UIPQ (see \cite[Eq. (8)]{ACpercopeel} where this quantity is denoted by $\delta^\square/2$).   The extra multiplicative factor $2$ in the time parameter comes from the fact that we
are dealing with half-perimeters. Similarly, the limiting process in the analog of Theorem \ref{thm:scalingfpp} is $ (\mathsf{p}_\square\cdot\mathcal{L}_{2\mathsf{p}_\square  t   }, \mathsf{v}_\square\cdot\mathcal{M}_{2\mathsf{p}_\square t })_{t \geq 0}$.

\bibliographystyle{siam}

\begin{thebibliography}{10}

\bibitem{AB14}
{\sc J.~Ambj{\o}rn and T.~Budd}, {\em Multi-point functions of weighted cubic
  maps}, arXiv:1408.3040.

\bibitem{Ang05}
{\sc O.~Angel}, {\em Scaling of percolation on infinite planar maps, {I}},
  arXiv:0501006.

\bibitem{Ang03}
\leavevmode\vrule height 2pt depth -1.6pt width 23pt, {\em Growth and
  percolation on the uniform infinite planar triangulation}, Geom. Funct.
  Anal., 13 (2003), pp.~935--974.

\bibitem{ACpercopeel}
{\sc O.~Angel and N.~Curien}, {\em Percolations on infinite random maps,
  half-plane models}, Ann. Inst. H. Poincar\'e Probab. Statist., 51 (2014),
  pp.~405--431.

\bibitem{AR13}
{\sc O.~Angel and G.~Ray}, {\em Classification of half planar maps}, Ann.
  Probab. (to appear).

\bibitem{AS03}
{\sc O.~Angel and O.~Schramm}, {\em Uniform infinite planar triangulation},
  Comm. Math. Phys., 241 (2003), pp.~191--213.

\bibitem{BCsubdiffusive}
{\sc I.~Benjamini and N.~Curien}, {\em Simple random walk on the uniform
  infinite planar quadrangulation: Subdiffusivity \emph{via} pioneer points},
  Geom. Funct. Anal., 23 (2013), pp.~501--531.

\bibitem{Ber96}
{\sc J.~Bertoin}, {\em L\'evy processes}, vol.~121 of Cambridge Tracts in
  Mathematics, Cambridge University Press, Cambridge, 1996.

\bibitem{BD94}
{\sc J.~Bertoin and R.~A. Doney}, {\em On conditioning a random walk to stay
  nonnegative}, Ann. Probab., 22 (1994), pp.~2152--2167.

\bibitem{BG09}
{\sc J.~Bouttier and E.~Guitter}, {\em Distance statistics in quadrangulations
  with a boundary, or with a self-avoiding loop}, J. Phys. A, 42 (2009),
  pp.~465208, 44.

\bibitem{Bud15}
{\sc T.~Budd}, {\em The peeling process of infinite {B}oltzmann planar maps},
  arXiv:1506.01590.

\bibitem{CC08}
{\sc F.~Caravenna and L.~Chaumont}, {\em Invariance principles for random walks
  conditioned to stay positive}, Ann. Inst. Henri Poincar\'e Probab. Stat., 44
  (2008), pp.~170--190.

\bibitem{CD06}
{\sc P.~Chassaing and B.~Durhuus}, {\em Local limit of labeled trees and
  expected volume growth in a random quadrangulation}, Ann. Probab., 34 (2006),
  pp.~879--917.

\bibitem{CurPSHIT}
{\sc N.~Curien}, {\em Planar stochastic hyperbolic triangulations}, Probab.
  Theory Related Fields (to appear), arXiv:1401.3297.

\bibitem{CurKPZ}
\leavevmode\vrule height 2pt depth -1.6pt width 23pt, {\em A glimpse of the
  conformal structure of random planar maps}, Commun. Math. Phys., 333 (2015),
  pp.~1417--1463.

\bibitem{CLGmodif}
{\sc N.~Curien and J.-F. Le~Gall}, {\em First-passage percolation and local
  modifications of distances in random planar maps}, in preparation.

\bibitem{CLGHull}
\leavevmode\vrule height 2pt depth -1.6pt width 23pt, {\em The hull process of
  the {B}rownian plane}, Probab. Theory Related Fields (to appear),
  arXiv:1409.4026.

\bibitem{CLGplane}
\leavevmode\vrule height 2pt depth -1.6pt width 23pt, {\em The {B}rownian
  plane}, J. Theoret. Probab., 27 (2014), pp.~1249--1291.

\bibitem{IL71}
{\sc I.~A. Ibragimov and Y.~V. Linnik}, {\em Independent and stationary
  sequences of random variables}, Wolters-Noordhoff Publishing, Groningen,
  1971.

\bibitem{Jeu82}
{\sc T.~Jeulin}, {\em Sur la convergence absolue de certaines int{\'e}grales},
  S{\'e}minaire de Probabilit{\'e}s XVI, Lecture Notes Math., 920 (1982),
  pp.~248--256.

\bibitem{Kri05}
{\sc M.~Krikun}, {\em Local structure of random quadrangulations},
  arXiv:0512304.

\bibitem{Kri04}
\leavevmode\vrule height 2pt depth -1.6pt width 23pt, {\em A uniformly
  distributed infinite planar triangulation and a related branching process},
  J. Math. Sci. (N.Y.), 131 (2005), pp.~5520--5537.

\bibitem{Kri07}
\leavevmode\vrule height 2pt depth -1.6pt width 23pt, {\em Explicit enumeration
  of triangulations with multiple boundaries}, Electron. J. Combin., 14 (2007),
  pp.~1--14.

\bibitem{LP10}
{\sc R.~Lyons and Y.~Peres}, {\em Probability on Trees and Networks}, Current
  version available at http://mypage.iu.edu/~rdlyons/, to appear.

\bibitem{Men08}
{\sc L.~M{\'e}nard}, {\em The two uniform infinite quadrangulations of the
  plane have the same law}, Ann. Inst. Henri Poincar{\'e} Probab. Stat., 46
  (2010), pp.~190--208.

\bibitem{MN13}
{\sc L.~M{\'e}nard and P.~Nolin}, {\em Percolation on uniform infinite planar
  maps}, Electron. J. Probab., 19 (2014), pp.~1--27.

\bibitem{MS13}
{\sc J.~Miller and S.~Sheffield}, {\em Quantum {L}oewner evolution},
  arXiv:1312.5745.

\bibitem{Pit06}
{\sc J.~Pitman}, {\em Combinatorial stochastic processes}, vol.~1875 of Lecture
  Notes in Mathematics, Springer-Verlag, Berlin, 2006.

\bibitem{Ray13}
{\sc G.~Ray}, {\em Geometry and percolation on half planar triangulations},
  Electron. J. Probab., 19 (2014), pp.~1--28.

\bibitem{St14}
{\sc R.~Stephenson}, {\em Local convergence of large critical multi-type
  {G}alton-{W}atson trees and applications to random maps}, arXiv:1412.6911.

\bibitem{Wat95}
{\sc Y.~Watabiki}, {\em Construction of non-critical string field theory by
  transfer matrix formalism in dynamical triangulation}, Nuclear Phys. B, 441
  (1995), pp.~119--163.

\end{thebibliography}

\end{document}